%% file: coalg.tex
\documentclass[reqno]{amsart}
\usepackage{verbatim}
\usepackage[textsize=scriptsize]{todonotes}
\usepackage{longtable}
\usepackage{bbm}
\usepackage{amscd}
\usepackage{stmaryrd}

\usepackage[margin=1in]{geometry}
\input{preamble2.tex}

\usepackage{tikz-cd}
\usepackage{mathtools}
\usepackage{bbm}

\def\cCAlg{\mathrm{cCAlg}}
\newtheorem{warning}[proposition]{Warning}

\renewcommand{\Mod}{\mathcal{M}\mathrm{od}}
\def\Pro{\mathrm{Pro}}

\newcommand{\PolyFun}{{\scr P\mathrm{oly}}}
\newcommand{\PolyFund}[1]{{\scr P\mathrm{oly}^{\leq #1}}}

\newcommand{\Sp}{\mathcal{S}\mathrm{p}}

\newcommand{\Set}{\mathcal{S}\mathrm{et}}
\newcommand{\ainj}{\mathrm{ainj}}
\newcommand{\ev}{\mathrm{ev}}
\def\Ss{\mathbb{S}}
\def\Id{\mathrm{Id}}

\def\DD{\mathcal{D}}
\def\one{\mathbbm{1}}

\def\img{\mathrm{Img}}
\def\coker{\mathrm{Coker}}

\def\nil{\mathrm{nil}}

\def\solvable{solvable}

\def\VectFk{\Mod_k^{\varphi}}

\newtoggle{final}
\toggletrue{final}

\iftoggle{final} {
\renewcommand{\todo}[1]{}
\newcommand{\NB}[1]{}
}{ 
\newcommand{\NB}[1]{\todo[color=gray!40]{#1}}
}

\numberwithin{proposition}{section}

\title{$\mathbb{E}_\infty$-coalgebras and $p$-adic homotopy theory}
\date{\today}

\author{Tom Bachmann}
\address{Mathematisches Institut, JGU Mainz, Germany}
\email{tom.bachmann@zoho.com}

\author{Robert Burklund}
\address{Department of Mathematical Sciences, University of Copenhagen, Denmark}
\email{rb@math.ku.dk}

\begin{document}

\maketitle

\begin{abstract}
We show that for any separably closed field $k$ of characteristic $p>0$, the canonical functor from nilpotent $p$-adic spaces to $\E_\infty$-coalgebras over $k$ (given by singular chains with coefficients in $k$) is fully faithful.
We also identify the essential image of simply connected spaces inside coalgebras.
This dualizes and removes finiteness assumptions from a theorem of Mandell.
\end{abstract}

\setcounter{tocdepth}{1}
\tableofcontents

\section{Introduction}

Beginning with the work of Quillen on rational homotopy theory \cite{QuillenRational} homotopy theorists have sought to construct algebraic models for as large a subcategory of the homotopy category of spaces as possible. The story begins with Quillen's pair of models for rational, simply connected spaces as either DG-Lie algebras or DG-coalgebras. Sullivan then constructed a dual model for rational, simply connected, finite type spaces as a DG-algebra of rational cochains \cite{DGMS}. In the $p$-adic setting Mandell provided a model for $p$-complete, nilpotent spaces with finite type $\F_p$-homology based on their $\E_\infty$-algebra of $\overline{\F}_p$-cochains \cite{mandell-e-infty}. Removing the finiteness hypothesis, but working with the much more restricted class of $v_n$-periodic spaces Heuts constructed a model of $v_n$-periodic spaces as Lie algebras in $T(n)$-local spectra \cite{Heutsvn}. 
Most recently, Yuan has given a model for simply connected, finite spaces in terms of their $\E_\infty$-algebra of spherical cochains together with a Frobenius trivialization \cite{AllenThesis}.

Roughly speaking, the existing algebraic models for homotopy theory can be divided into two groups: cochain based models where a strong finiteness assumption is necessary and Lie algebra/coalgebra models where no such hypothesis is necessary. Cochains being the linear dual of chains it is natural, even expected, that some finiteness hypothesis is necessary in any approach based on cochains. On the other hand, while Lie algebra based models have been very successful when working at a fixed chromatic height, the fact that the $\F_p$ cochain algebra of a finite space is formally etale contravenes any hope for a full $p$-adic Lie algebra model. 
In this paper we show that the naive idea of simply working with chain coalgebras (as opposed to cochain algebras) is sufficient to obtain an algebraic model for $p$-adic homotopy theory without finiteness hypotheses.


\begin{definition}
  A space $X$ is $p$-complete if it is local with respect to the functor $\F_p \otimes \Sigma_+^\infty (-)$.
  We write $\Spc_p \subseteq \Spc$ for the subcategory of $p$-complete spaces.
  A space $X$ is nilpotent if its fundamental group is nilpotent and acts nilpotently on the higher homotopy groups (for every choice of base-point).\footnote{Note that for us nilpotent does not require connected.}
  We write $\Spc^{\nil} \subseteq \Spc$ for the subcatgegory of nilpotent spaces.
  We let $\Spc_p^{\nil} \coloneqq \Spc_p \cap \Spc^{\nil}$.
\end{definition}

Let $k$ be a field of characteristic $p$.
Equipping $\Spc$ with its cartesian symmetric monoidal structure, by the K\"unneth theorem the $k$-chains functor
\[ k[-]: \Spc \to \Mod_k \]
is symmetric monoidal.
As $\Spc$ is \emph{cartesian} symmetric monoidal, the forgetful functor
$\cCAlg(\Spc) \to \Spc$ is an equivalence \cite[Corollary 2.4.3.10]{HA},
therefore we obtain a refinement 
\[ k[-]: \cCAlg(\Spc) \to \cCAlg(\Mod_k). \]
of the $k$-chains functor landing in $\E_\infty$-$k$-coalgebras.

\begin{theorem}[Fully faithfulness theorem, see Theorem \ref{thm:main}] \label{thm:chains-ff}
  Let $k$ be a separably closed field of characteristic $p > 0$.
  The chains functor restricted to $p$-complete, nilpotent spaces
  \[ k[-] \colon \Spc_p^{\nil} \to \cCAlg_k \]
  is fully faithful.
\end{theorem}

After further restricting to simply connected $p$-complete spaces we are also able to identify the essential image of $k[-]$.
Let $\varphi : \Spec(k) \to \Spec(k)$ be the Frobenius endormorphism.
If $E \in \cCAlg_k$, then the $k$-vector spaces $\pi_i(E)$ acquire a Frobenius operation $\mathfrak{F}: \pi_i(E) \to \varphi^*\pi_i(E)$.
We say that $\pi_i(E)$ is \emph{\solvable} if the $k$-linear span of $\{ v \in \pi_i(E) \ |\ \mathfrak{F}(v) = 1 \otimes v \}$ is all of $\pi_i(E)$.

\begin{theorem}[Characterization theorem, see Theorem \ref{thm:charact}] \label{thm:chains-charct}
  Let $k$ be a separably closed field of characteristic $p > 0$, and $E \in \cCAlg_k$.
  There exists a simply connected, $p$-complete space $X$ with $k[X] \wequi E \in \cCAlg_k$ if and only if the following conditions hold:
  \begin{enumerate}
  \item $E$ is connective,
  \item $E$ is simply connected in the sense that $\pi_0(E) \cong k$ and $\pi_1(E) \cong 0$ and 
  \item $\pi_i(E)$ is \solvable{} for all $i \geq 2$.
  \end{enumerate}
\end{theorem}

Our result is closely related to Mandell's work \cite{mandell-e-infty}.
He proves that the \emph{cochains functor} \[ k^{(-)}: \Spc^\op \to \CAlg_k \] is fully faithful, when restricted to $p$-complete, nilpotent spaces with degreewise finite dimensional $\F_p$-homology.
The conjecture that it should be possible to remove the finiteness assumptions from Mandell's theorem by working with coalgebras instead of algebras appears to be folklore. 


\begin{remark}
  Amplifying the claim from the beginning of this section about the necessity of a finiteness assumption when working with cochains we consider the following example.
  Let $S$ be a set. Then $\Map_{\CAlg_k}(k^S, k)$ is just the set of maps of ordinary $k$-algebras from $k^S$ to $k$.
  This in turn may be identified with the Stone--Čech compactification of $S$ and can therefore be identified with $S$ exactly when $S$ is finite.
\end{remark}

Our proof proceeds along similar lines as Mandell's proof.
We first use an unstable $\F_p$-Adams resolution to reduce to the case of an Eilenberg--Maclane space corresponding to an $\F_p$-vector space.\footnote{Mandell uses induction on a principalized Postnikov tower with layers $\F_p$ or $\Z_p$, and treats these two cases separately. In absence of finiteness conditions there seems to be no reasonable classification of ($p$-complete) $\Z_p$-modules (which appear in the layers of a general nilpotent $p$-adic space), so this approach cannot quite work in our setting.}
Mandell then provides an explicit model for the cochains on $K(\F_p,n)$ as a pushout of free $\E_\infty$-algebras and works his way forward from there.
We refer to this pushout square as the \emph{Artin--Schreier square}.
For us the key point is then to construct a corresponding Artin--Schreier pullback square expressing the chains on $K(V,n)$ as a pullback of cofree $\E_\infty$-coalgebras. The difficulty here is that in order to remove the finiteness restriction we must allow $V$ to range over all $\F_p$-vector spaces (including, crucially, non-finite-dimensional $V$). Our proof proceeds in three steps:
\begin{enumerate}
\item We use the dual of Mandell's arguments to give a pro-object version of the Artin--Schreier square for finite dimensional $V$.
\item Analyzing the way that the pro-objects from step (1) vary with $V$ we observe that their dependence on $V$ is \emph{polynomial}.
\item Using polynomiality we show that all pro-objects we have encountered have pro-constant homotopy groups.
  Pro-constancy allows us to materialize without issue and obtain the Artin--Schreier pullback square.
\end{enumerate}
It is in steps (2) and (3) that our work diverges from that of Mandell. 



\subsection*{Outline}
We begin in \S\ref{sec:coalg} by defining coalgebras recalling some of their basic properties.
Then in \S\ref{sec:frob} we construct a family of maps which we call \emph{coalgebra Frobenii}.
Using these maps we construct the Artin--Schreier squares.

The bulk of the paper is then devoted to proving that these Artin--Schreier squares are cartesian.
To this end, we need some preparations.
In \S\ref{set:polyfun} we introduce the notion of polynomial functors on finite pointed sets and almost-injective maps, and establish some basic properties.
In \S\ref{sec:pro} we recall facts about categories of pro-objects, and show how they can be used to overcome certain technical difficulties regarding coalgebras.
The final preparatory section \S\ref{sec:idem} establishes some facts about localization of coalgebras at idempotents.
In \S\ref{sec:key-lemma} we deliver the proof described above, proving that the Artin--Schreier squares are cartesian.
Finally, in the more straightforward \S\ref{sec:main} we establish our main theorems as corollaries of the results of \S\ref{sec:key-lemma}.


\subsection*{Acknowledgements}
We would like to thank Shaul Barkan, Mike Hopkins, Marc Hoyois, Florian Riedel and Jan Steinebrunner for helpful comments.

We also gratefully acknowledge the immense debt our work pays to Mandell's article \cite{mandell-e-infty}.
Our use of Artin--Schreier squares is entirely dual to his.
Attempting to construct and prove cartesian these squares was the main idea of our article.

The authors acknowledge support by Deutsche Forschungsgemeinschaft (DFG, German Research Foundation) through the Collaborative Research Centre TRR 326 \textit{Geometry and Arithmetic of Uniformized Structures}, project number 444845124.
During the course of this work the second author was supported by NSF grant DMS-2202992 and by the DNRF through the Copenhagen center for Geometry and Topology (DNRF151).

\subsection*{Notation and Conventions}

\begin{enumerate}
\item Category means $\infty$-category.
\item Given a symmetric monoidal category $\scr C$, we denote by
  \[ \cCAlg(\scr C) \coloneqq \CAlg(\scr C^\op)^\op \]
  its category of cocommutative coalgebras.
\item We say that a stable, symmetric monoidal category is stably symmetric monoidal if the tensor product commutes with finite (co)limits separately in each variable.
\item We write $\Spc$ for the category of spaces, $\Spc_*$ for the category of pointed spaces and use superscripts $\mathrm{fd}$ and $\mathrm{fin}$ for the subcategories of finite dimensional and finite spaces respectively.
\item We write $\Set_*^{\ainj}$ (respectively $\Fin_*^{\ainj}$) for the categories of (finite) pointed sets and almost-injective maps (see \Cref{def:ainj}).
\item Throughout we fix a field $k$ of characteristic $p$ and write $\varphi : k \to k$ for the Frobenius endomorphism of $k$.
\item We denote by $\Mod_k$ the category of $k$-modules in $\Sp$, the category of spectra.
  The $1$-category of $k$-vector spaces is denoted $\Mod_k^\heart$.
\item The functor $\varphi^*: \Mod_k \to \Mod_k$ is scalar extension along $\varphi: k \to k$.
  In other words $\varphi^*(V) = \tilde k \otimes_\varphi V$, where $\tilde k$ denotes $k$ with $k$-module structure via $\varphi$, and the $k$-module structure on $\varphi^*(k)$ is via the left hand factor.
\item We fix a regular cardinal $\kappa$ such that the cofree coalgebra functor
  $ C_k : \Mod_k \to \cCAlg_k $ is $\kappa$-accessible (see \S\ref{sec:coalg}).
\item Given $X \in \Spc_*$ we let $k\{S\} \coloneqq k \otimes \Sigma^\infty X$.
  When $X$ is a pointed set $k\{X\}$ is the $k$-vector space with generators $X$ and relation $*=0$.
\item Given $X \in \Spc$ we put $k[X] = k \otimes \Sigma^\infty_+ X$.
\end{enumerate}

We write $k[X]$ for both the $k$-module and the $k$-coalgebra.
Since the forgetful functor $U: \cCAlg_k \to \Mod_k$ does not preserve limits, this notation can be ambiguous.
We sometimes write $Uk[X]$ to emphasise that we are thinking of the module, not the coalgebra.

\section{Coalgebras}
\label{sec:coalg}


Let $\scr C$ be a stably symmetric monoidal category.
We write
\[ \cCAlg(\scr C) \coloneqq \CAlg(\scr C^{\op})^{\op} \]
for the category cocommutative coalgebras $\scr C$
and
\[ U_{\scr C} : \cCAlg(\scr C) \to \scr C \]
for the associated forgetful functor.
The functor $U_{\scr C}$ is symmetric monoidal where $\cCAlg(\scr C)$ is given its cartesian symmetric monoidal structure \cite[Proposition 3.2.4.7]{HA}.
In particular the underlying object of a product of coalgebras is given by their tensor product in $\scr C$.
  
Whenever $U_{\scr C}$ admits a right adjoint we will denote it by $C_{\scr C}$ and call it the \emph{cofree coalgebra functor}.
We write $\epsilon : U_{\scr C}C_{\scr C}(X) \to X$ for the counit map of this adjunction.
The cofree coalgebra functor exists in the following situations:
\begin{itemize}
\item If $\scr C$ is presentable and $- \otimes -$ is accessible, then $\cCAlg(\scr C)$ is presentable and $U_{\scr C}$
  admits a right adjoint \cite[Corollaries 3.1.4 and 3.1.5]{Elliptic-I}.
\item If $\scr C^{\op}$ is presentably symmetric monoidal then $\cCAlg(\scr C)$ is presentable and $U_{\scr C}$
  admits a right adjoint \cite[Corollary 3.2.3.5 and Example 3.1.3.6]{HA}.
\end{itemize}  
  
When $\scr C$ is clear from context we will drop it from the notation.
In the case $\scr C = \Mod_k$ we will write $C_k \coloneqq C_{\Mod_k}$.

\begin{warning} \label{warning:cofree}
  Since the tensor product in $\scr C$ usually does not commute with cosifted limits, the cofree coalgebra is usually \emph{not} given by the ``expected formula''
  $\prod_i ((\ph)^{\otimes i})^{h \Sigma_i}$.
\end{warning}

\begin{definition}
  Let $\scr C$ be a stably symmetric monoidal category which admits countable limits and colimits.
  Given an object $X \in \scr C$ we write
  \[ T_{\scr C}(V) = (V^{\otimes p})^{tC_p} \]
  for the \emph{Tate construction} (see e.g. \cite[Definition 2.1.23]{lurie2011derived}).
\end{definition}

Computing cross-effects we learn that $T_{\scr C}(-)$ is an exact functor \cite[Proposition 2.2.3]{lurie2011derived}.

\begin{lemma} \label{lem:tate-trivial}
  Let $\scr C$ be a $k$-linear category with countable limits and colimits.
  There is a natural isomorphism 
  \[ X^{tC_p} \cong \left( \prod_{0 > j} \Sigma^jX \right) \oplus X \oplus \left( \bigoplus_{j > 0} \Sigma^jX \right) \]
  between the Tate construction on an object with trivial action and the indicated expression.
  We will write $\mathrm{pr}_0 : X^{tC_p} \to X$ associated projection operation.  
\end{lemma}

\begin{proof}
  As the action is trivial and $\scr C$ is $k$-linear we can identify
  $X^{hC_p}$ with $X^{BC_p}$ and $X_{hC_p}$ with $X \otimes BC_P$.
  The norm map is then equivalently given by a natural transformation $X \to X$ in $\scr C^{B(C_p \times C_p)}$. Unrolling definitions (see \cite[Construction I.1.7]{NS}) this natural transformation is given by multiplication by $[(C_p \times C_p)/\Delta] = p = 0$ in $\pi_0k^{B(C_p \times C_p)}$. It follows that we may choose a null homotopy of the norm map. Choosing a splitting of $BC_p$ in $\Mod_k$ now completes the proof of the lemma.
\end{proof}

\section{The Frobenius} \label{sec:frob}

In this section we consider Frobenius endomorphisms for coalgebras.
Using such an endomorphism we then construct, for every $\F_p$-vector space $V$, a certain commutative square involving $k[K(V, n)]$, the cofree coalgebra on $\Sigma^n k \otimes V$, and the Frobenius endomorphism.
The majority of the remainder of the article will be concerned with proving that this square is cartesian.

\begin{lemma} \label{lem:T-homotopy}
  For each $i \in \Z$ there is a natural isomorphism
  between $\pi_iT_k(V)$ and the set of Laurent series
  \[ \left\{ \sum_{j=\epsilon + 2m} v_je^{\epsilon}t^m \ |\ v_j \in \varphi^*\pi_{i+j}(V),\quad \epsilon \in \{0,1\} \right\}. \]
\end{lemma}

\begin{proof}
  At the $1$-categorical level $V$ naturally splits as the sum of its homotopy groups, $V_j$.
  Using assembly and coassembly maps we have natural comparisons
  \[ T_k\left( \bigoplus_j V_j \right) \to \left( \prod_{j \geq 0} T_k( V_j ) \right) \oplus T_k\left( \bigoplus_{j < 0} V_j \right) \leftarrow \left( \prod_{j \geq 0} T_k( V_j ) \right) \oplus \left( \bigoplus_{j<0} T_k( V_j ) \right). \]  
  After picking a basis for $V$ we have associated diagonal and co-diagonal maps
  $V \to V^{\otimes p} \to V$ such that the complement of $V$ in $V^{\otimes p}$ has a free $C_p$ action.
  This lets us read off that the comparison maps above are isomorphisms.

  Using the structure of $\pi_* T_k(V)$ as a module over $\pi_* T_k(k)$, what remains is to give a natural isomorphism $\pi_0 T_k(V) \wequi \varphi^* V$
  for $V$ concentrated in degree $0$.
  Note that $\pi_0 T_k(V) \iso (V^{\otimes p})^{C_p}/(\text{cross terms})$.
  The canonical map $V \to \varphi_*((V^{\otimes p})^{C_p}), v \mapsto v^{\otimes p}$ is compatible with scalar multiplication, and becomes $k$-linear after projection to the quotient $\varphi_* \pi_0 T_k(V)$.
  By adjunction we obtain a natural map $\varphi^* V \to \pi_0 T_k(V)$.
  We can observe it is an isomorphism by picking a basis for $V$.
\end{proof}

\begin{construction} \label{cntr:pi-star-frob}
  Given a $k$-coalgebra $R$ the universal stable arity $p$ co-operation is the natural map
  \[ R \xrightarrow{\Delta} (R^{\otimes p})^{hC_p} \to (R^{\otimes p})^{tC_p} \cong T_k(R). \]
  Taking homotopy groups, passing across the isomorphism from \Cref{lem:T-homotopy}
  and extracting the coefficient of $t^0$ we obtain a natural $k$-linear map
  \[ \mathfrak{F} : \pi_i(R) \to \varphi^* \pi_i(R). \]
  We refer to this map as the coalgebra Frobenius.
\end{construction}

\begin{definition}
  The functor $\varphi^*$ provides an action of $\N$ on $\Mod_k$ and we let $\VectFk$ denote the oplax limit of this action.
  We refer to the objects of this category as \emph{oplax $\varphi$-modules}.
  Concretely, an oplax $\varphi$-module consists of a $k$-module $V$ and a $k$-linear map $\alpha :  V \to \varphi^*V$.
\end{definition}

The coalgebra Frobenius equips the homotopy groups of any $R \in \cCAlg_k$ with the structure of
an \emph{oplax $\varphi$-module}. One of the main goals of this section will be to isolate the subcategory of oplax $\varphi$-modules which appear as homotopy groups of chain coalgebras.


The $k$-vector space $k$ has a preferred oplax $\varphi$-module structure where $\alpha = \Id$.
Tensoring this object with an $\F_p$-module and mapping out of this object together yield an adjunction
\[ \begin{tikzcd} \VectFk \ar[r, bend right, "(-)^0"] & \Mod_{\F_p} \ar[l, bend right, "k \otimes -"] \end{tikzcd} \]
where the functor $(-)^0$ is given by
\[ (V,\alpha)^0 \coloneqq \map_{\VectFk}(k, (V,\alpha)) \cong  \mathrm{fib}( V \xrightarrow{\mathrm{can} - \alpha}  \varphi^*V). \]

\begin{lemma} \label{lem:solvable-ff}
  If $k$ is separably closed, then the functor $k \otimes -$ is fully faithful.
\end{lemma}

\begin{proof}
  As both $k \otimes -$ and $(-)^0$ preserve colimits and every $\F_p$-vector space is a sum of shifts of $\F_p$ itself it will suffice to show that the
  unit map $\F_p \to k^0$ is an isomorphism. Unrolling the definitions this reduces to the exactness of the Artin--Schreier sequence
  \[ 0 \to \F_p \to k \xrightarrow{x \mapsto x-x^p} k \to 0. \]
\end{proof}

\begin{definition} \label{def:solvable}
We say that an oplax $\varphi$-module is \emph{\solvable} if it is in the essential image of $k \otimes -$.
\end{definition}

\begin{lemma}
  Let $(V, \alpha)$ be an oplax $\varphi$-module with $V \in \Mod_k^\heart$.
  $(V, \alpha)$ is \solvable{} if and only if 
  the $k$-linear span of $\{ v \in V \ |\ \alpha(v) = 1 \otimes v \} = \pi_0((V,\alpha)^0)$ is all of $V$.
\end{lemma}

\begin{proof}
  Observe that we have a natural map $k \otimes \pi_0((V,\alpha)^0) \to V$.
  It will suffice to prove that this map is injective.
  Our proof follows that of \cite[Lemma 3.5.3]{lurie2011derived} closely.

  Suppose we have $v_1, \cdots, v_n \in V$ with $\alpha(v_i) = 1 \otimes v_i$ which are linearly independent over $\F_p$
  and $\sum_{i=1}^n \lambda_i v_i = 0$ with $\lambda_i \in k$.
  We must show that $\lambda_1 = \cdots = \lambda_n = 0$. We proceed by contradiction.
  Assume $n$ is minimal for such a relation (so that all $\lambda_i$ are non-zero) and dividing by $\lambda_1$ we may even assume that $\lambda_1 = 1$.
  Applying $\alpha$ and $1 \otimes -$ to this relation and subtracting we find that
  \[ 0 = \sum_{i=1}^n (\lambda_i - \lambda_i^p) \otimes v_i = \sum_{i=2}^n (\lambda_i - \lambda_i^p) \otimes v_i \]
  in $\varphi^*V$. As $\varphi^*(-)$ is exact and only $v_2, \cdots, v_n$ appear in the latter expression this contradicts the minimality of $n$.
\end{proof}

We now turn to the task of constructing variants of the coalgebra Frobenius which exist at the spectrum and coalgebra level.

\begin{lemma} \label{lem:untwist2}
  The lax symmetric monoidal natural transformation
  \[ \Delta^{tC_P} : (k\{-\})^{tC_p} \Rightarrow (k\{(-)^{\wedge p}\})^{tC_p} : \Spc_* \to \Mod_k. \]  
  induced by the diagonal is an isomorphism on finite dimensional spaces.
\NB{rwb: After thinking about it carefully, $\coprod_{\infty} BC_2$ is a counterexample without the finite dimensionality assumption.}
\end{lemma}

\begin{proof}
  As the source and target both preserve finite colimits we may reduce to the case where $X \in \Set_*$.
  For a pointed set $X$ the diagonal map $X \to X^{\wedge p}$ induces a splitting of
  $X^{\wedge p}$ as $X \vee Y$ where the $C_p$-action on $Y$ is free.
  As $(-)^{tC_p}$ kills free summands, the desired conclusion now follows.  
\end{proof}

\begin{construction} \label{def:artin-schreier}
  Let $X \in \Spc_*^{\mathrm{fd}}$.
  We construct the \emph{Frobenius functional}
  $ F_X : C_k(k\{X\}) \to k\{X\} $
  as the composite of the canonical arity $p$ co-operation
  \[  C_k(k \{X\}) \xrightarrow{\Delta} (C_k(k \{X\})^{\otimes p})^{hC_p} \xrightarrow{\epsilon} ((k \{X\})^{\otimes p})^{hC_p} \]
  on a cofree coalgebra with the projection  
  \[ ((k \{X\})^{\otimes p})^{hC_p} \to T_k(k \{X\}) \stackrel{L.\ref{lem:untwist2}}{\cong} (k \{X\})^{tC_p} \xrightarrow{\mathrm{pr}_0} k \{X\}. \]
  Passing to mates we obtain the \emph{Frobenius map}
  \[ F \colon C_k(k \{X\}) \to C_k(k \{X\}) \]
  as mate to $F_X$
  and the \emph{Artin--Schreier map}
  \[ 1-F \colon C(k \{X\}) \to C(k \{X\}) \]
  as mate to $\epsilon - F_X$.
  (Recall that $\epsilon$ is the counit of the underlying--cofree adjunction).
\end{construction}

Note that by construction, both $F$ and $1-F$ are natural in $X \in \Spc_*^{\mathrm{fd}}$.

\begin{remark} \label{rmk:Frob-Q0}
  Restricting to the objects $S^n \in \Spc_*$ the
  coalgebras $C_k(k \otimes \Ss^n)$ represent the
  functor sending a coalgebra $A$ to the underlying space of its $k$-linear dual, $A^\vee$ (and loop spaces of this space).
  The $k$-linear dual $A^\vee$ moreover naturally lands in $\CAlg_k$.
  Unrolling the definitions of the Frobenius and Artin--Schreier maps
  we may identify their action on the underlying space of the commutative algebra $A^\vee$
  as the Dyer--Lashof operations $Q_0$ and $1-Q_0$ respectively.
\end{remark}

\begin{remark} \label{rmk:F-vs-F}
  Comparing the Constructions \ref{cntr:pi-star-frob} and \ref{def:artin-schreier} we can extract that
  the map $F : (\pi_iR)^{\vee} \to (\pi_iR)^{\vee} $ obtain by specializing \ref{def:artin-schreier} to the case $X=S^i$ and looking at maps out of $R$ and
  the map $\mathfrak{F} : \pi_iR \to \varphi^* \pi_iR $ are related by the following diagram:
  \[ \begin{tikzcd}
    & (\varphi^* \pi_iR)^{\vee} \ar[dr, "\mathfrak{F}^{\vee}"] & \\
    (\pi_iR)^\vee \ar[rr, "F"] \ar[ur, "f \mapsto \varphi^*f"] & & (\pi_*R)^\vee.
  \end{tikzcd} \]  
\end{remark}
\todo{Think about this.}

As a consequence of the correspondence between Dyer--Lashof operations and Steenrod operations on cochain algebras and the fact that the bottom Steenrod operation acts by the identity (see e.g. \cite[Proposition 8.1]{may-algebraic-approach}) we obtain:

\begin{corollary} \label{cor:pi-star-solvable}
  Given $X \in \Spc$, the oplax $\varphi$-module $\pi_i(k[X])$ is \solvable.
\end{corollary}



\begin{construction} \label{cnstr:AS-square}
  We will construct a functor
  $\Spc_*^{\mathrm{fd}} \to \Fun((\Delta^1)^{\times 2}, \cCAlg_k)$
  sending a pointed space $X$ to a commutative square of coalgebras
  \[ \begin{tikzcd}
    k[\Omega^\infty (\F_p \{X\})] \ar[r] \ar[d] &
    C_k(k \{X\}) \ar[d, "1-F"] \\ 
    k \ar[r] &
    C_k(k \{X\}),
  \end{tikzcd} \]
  where
  \begin{enumerate}
  \item the top horizontal map is
    the mate of
    the $k$-linearization of
    the mate of
    $\Omega^\infty$ of
    the map of spectra
    \[ \F_p \{X\} \to k \{X\}. \]
  \item the bottom horizontal map is obtained by applying $C_k$ to the map
    $0 \to k \{X\}$,
  \item the left vertical map is the unique one ($k$ being the terminal coalgebra), and
  \item the right vertical map is the Artin--Schreier map of \Cref{def:artin-schreier}.
  \end{enumerate}
\end{construction}

\begin{proof}[Details.]
  What remains is to construct a homotopy filling the square above, natural in $X$.
  For this we start by analyzing the space of of natural transformations
  \[ \Map_{\Fun(\Spc_*^{\mathrm{fd}}, \cCAlg_k)} (k[\Omega^\infty (\F_p \{-\})], C_k(k \{-\})). \]
  Passing to mates twice we see this is isomorphic to
  \[ \Map_{\Fun(\Spc_*^{\mathrm{fd}}, \Spc)}(\Omega^\infty (\F_p \{-\}), \Omega^\infty(k \{-\})). \]
  Using the fact that every finite dimensional space $X$ can be written as a filtered colimit of finite spaces 
  and both functors are defined on all spaces and commute with filtered colimits on this category we see that the space above is isomorphic to
  \[ \Map_{\Fun(\Spc_*^{\mathrm{fin}}, \Spc)}(\Omega^\infty (\F_p \{-\}), \Omega^\infty(k \{-\})). \]
  Applying \cite[Proposition 1.4.2.22]{HTT} (using that both functors are reduced and excisive) we get an isomorphism with
  \[ \Map_{\Fun(\Spc_*^{\mathrm{fin}}, \Sp)}(\F_p \{-\}, k \{-\}). \]
  Finally, because both functors preserve finite colimits, and 
  $ \Fun^{\mathrm{Lex}}(\Spc_*^{\mathrm{fin}}, \Sp) \iso \Fun^{\mathrm{Lex}}(\Spc^{\mathrm{fin}}, \Sp) \iso \Fun(*, \Sp) \iso \Sp $
  this last space is isomorphic to $\Map_{\Sp}(\F_p, k)$.
  
  Unrolling these isomorphisms we see that the space of choices of $2$-cell in the square is either empty or contractible and that it is nonempty if and only if the original square commutes when $X=S^0$.
  In the case $X=S^0$ the claim reduces to verifying that $F$ acts as the identity on
  $k[\Omega^\infty(\F_p \{S^0\})]$.
  As $\Omega^\infty( \F_p)$ is a set with $p$ elements
  this in turn reduces to evaluating $F$ on $k$ itself where it acts as the identity (see \Cref{rmk:Frob-Q0}).
\end{proof}

Precomposing the functor from \Cref{cnstr:AS-square} with the functor
$ \Set_* \hookrightarrow \Spc_* \xrightarrow{\Sigma^n} \Spc_* $
we obtain a commutative square in $\cCAlg(\Mod_k)$ of the form
\begin{equation} \label{eq:key-pullback}
\begin{CD}
k[\Omega^\infty \Sigma^n \F_p\{W\}] @>c>> C_k(\Sigma^n k\{W\}) \\
@VVV                             @V{F-1}VV \\
k @>{0}>>    C_k(\Sigma^n k\{W\})
\end{CD}
\end{equation}
natural in $W \in \Set_*$.
We shall prove in \S\ref{sec:key-lemma} that this is a pullback square.
This is key to the proof of our main theorem.

\section{Polynomial Functors} \label{set:polyfun}

In preparation for proving that the Artin--Schreier square from \Cref{eq:key-pullback} is a pullback square we will need to review some material on polynomial functors.

\begin{definition} \label{def:ainj}
  Let $\Set_*^\ainj$ be the category of pointed sets and almost-injective maps, that is, pointed maps $f: S \to T$ such that
  the restriction of $f$ to $f^{-1}(T \setminus *)$ is injective.
\end{definition}

We shall study $\Fun(\Set_*^\ainj, \Mod_k^\heart)$.

\begin{construction}
  Let $D$ be the exact functor
  \[ D : \Fun(\Set_*^\ainj, \Mod_k^\heart) \to \Fun(\Set_*^\ainj, \Mod_k^\heart) \]
  defined by $D(F)(X) = \ker(F(X \vee S^0) \to F(X \vee *))$.
\end{construction}

To see that $D$ is exact note that there is a retraction
$ F(X \vee *) \to F(X \vee S^0) \to F(X \vee *) $.
In particular we have $F(- \vee S^0) \cong F(-) \oplus D(F)(-)$.

\begin{definition} \label{dfn:polynomial}
  Given a functor $F \in \Fun(\Set_*^\ainj, \Mod_k^\heart)$
  we say that it is \emph{polynomial of degree $\leq d$} if
  \begin{enumerate}
  \item $F$ commutes with filtered colimits,
  \item $F$ takes finite sets to finite dimensional vector spaces and
  \item $D^{d+1}(F) = 0$.
  \end{enumerate}
  We write $\PolyFund{d}$ for the category of polynomial functors of degree $\leq d$
  and $\PolyFun$ for the category of all polynomial functors (that is, functors which are polynomial of some degree).
\end{definition}

\begin{example} \label{ex:poly-free}
  Let $\scr V$ be the functor $X \mapsto k\{X\}$, where $k\{X\}$ is the vector space with generators $X$ and relation $*=0$.
  $\scr V$ is polynomial of degree $1$.
\end{example}

\begin{example} \label{ex:proj-gens}
  Let $X$ be a finite pointed set.
  The $k$-linearized corepresentable functor on $X$ given by $k[ \Map(X,-) ]$
  is projective and polynomial of degree $|X|-1$.
  In fact these examples are somewhat redundant as $k[\Map(S^0,-)] \cong k \oplus \scr V $.
\end{example}

A polynomial functor is determined by its restriction to the subcategory of finite pointed sets and almost injective maps.
This means that the projective objects from \Cref{ex:proj-gens} are a collection of projective generators for $\PolyFun$.

\begin{lemma} \label{lem:poly-closure}
  The full subcategory
  $\PolyFun \subseteq \Fun(\Set_*^\ainj, \Mod_k^\heart)$
  of polynomial functors is closed under
  extensions, passing to subobjects and quotients and (pointwise) tensor products.
\end{lemma}

\begin{proof}
  It is clear that the subcategory of functors which taking values in finite dimensional vector spaces is closed under extensions, passing to subobjects and quotients and tensor products. We may therefore focus on the vanishing of $D^n$.
  As $D(-)$ is exact the vanishing of $D^n(F)$ implies the same for any subobject or quotient.
  Similarly, if $D^n(F)$ and $D^n(G)$ vanish so does any extension of them.
  For tensor products we note that $F(\ph \vee S^0) \cong F \oplus DF$.
  From this it follows that
  \[ D(F \otimes G) \cong F \otimes D(G) \oplus D(F) \otimes G \oplus D(F) \otimes D(G), \]
  and hence $D^{n+m}(F \otimes G) = 0$ if $D^n F = 0$ and $D^m G = 0$.
\end{proof}

\begin{lemma} \label{lem:der-conservative}
  The functor
  \[ D \times \ev_\emptyset : \PolyFun \to \PolyFun \times \Mod_k^\heartsuit \]
  is conservative.
\end{lemma}

\begin{proof}
  As $D$ and $\ev_\emptyset$ are each exact it will suffice to show that
  the kernel of $D \times \ev_\emptyset$ consists of only the zero object.
  Let $F$ be an object of $\PolyFun$ with
  $D(F) = 0$ and $F(\emptyset)=0$. We will show that $F=0$.
  Using the assumption that $D(F)=0$ we learn that
  $F(V) \to F(S^0 \vee V)$ is an isomorphism for every $V$.
  In particular, using that $F(\emptyset)=0$ and $F$ commutes with filtered colimits
  it now follows that $F(V) = 0$ for all $V$. 
\end{proof}

The key structural result we will need about polynomial functors on $\Set_*^\ainj$ is the following:

\begin{proposition} \label{cor:polyfun-artinian}
  The category of polynomial functors
  $\PolyFun$
  is noetherian and artinian.
\end{proposition}

\begin{proof}
  We will show that the subcategories
  $\PolyFund{d}$
  are noetherian and artinian by induction on $d$.
  As these subcategories exhaust $\PolyFun$
  and are closed under passing to subobjects this will suffice.
  The base-case $d=-1$ asserts that the zero category is noetherian and artinian.
  For the inductive step we note that
  the functor from \Cref{lem:der-conservative} restricts to an exact, conservative functor
  \[ D \times \ev_\emptyset : \PolyFund{d} \to \PolyFund{d-1} \times \Mod_k^{\omega,\heartsuit}. \]
  Exact, conservative functors induce inclusions on lattices of subobjects.
  The desired conclusion now follows from the fact that $\Mod_k^{\omega,\heart}$ is noetherian and artinian.
\end{proof}

\section{Pro-objects}
\label{sec:pro}

For a category $\scr C$ with finite limits, we denote by $\Pro(\scr C)$ the category of pro-objects in $\scr C$.
See \cite[\S3.1]{lurie2011derived} for a definition when $\scr C$ is accessible.
If $\scr C$ is in addition small, then by \cite[Example 3.1.2]{lurie2011derived} we have \[ \Pro(\scr C) \cong \Ind(\scr C^\op)^\op, \] and this is the form in which we will mainly work.
Beware that if $\scr C$ is large, then formation of $\Pro(\scr C)$ will require us to extend the universe.
In particular, if $\scr C$ is presentable, then $\Pro(\scr C)$ usually is not.

We think of objects of $\Pro(\scr C)$ as cofiltered families of objects of $\scr C$.
Finite limits and colimits in $\Pro(\scr C)$ are computed ``indexwise'' \cite[\S2.1]{MR4637146}, whereas cofiltered limits are computed ``formally'' (essentially by construction).
If $\scr C$ is stable then so is $\Pro(\scr C)$ \cite[Lemma 2.5]{MR4637146}.
Taking an object of $\scr C$ to the constant family yields a functor $c: \scr C \to \Pro(\scr C)$.
Mapping spaces in $\scr C$ are computed in the classical way \cite[\S2.1]{MR4637146}; in particular $c$ is fully faithful.
The functor $c$ has a partially defined right adjoint $M$.
If $\scr C$ has cofiltered limits then $M$ is given by taking the limit of the family.\footnote{$M$ preserves limits and, $c$ being fully faithful, sends constant families to their limits.}

Suppose that $\scr C$ is symmetric monoidal.
Then $\Pro(\scr C)$ inherits a symmetric monoidal structure in which $\otimes$ commutes with cofiltered limits \cite[Remark 2.4.2.7, Proposition 4.8.1.10]{HA}.
If $\otimes$ in $\scr C$ commutes with finite limits in each variable (e.g. $\scr C$ stable and $\otimes$ commutes with finite colimits), then it follows that the tensor product in $\Pro(\scr C)$ has the same property.
Hence in this case the tensor product on $\Pro(\scr C)$ commutes with \emph{arbitrary} limits in each variable separately.
This is the main reason why we consider pro-objects.
For example, we get the following.

\begin{lemma} \label{lemm:cofree-pro}
  Suppose that $\scr C$ is a symmetric monoidal category with finite limits, in which $\otimes$ commutes with finite limits in each variable.
  The forgetful functor
  $U : \cCAlg(\Pro(\scr C)) \to \Pro(\scr C)$ has a right adjoint $C$
  and there is a natural isomorphism   
  \[ UC(V) \cong \prod_{n \ge 0} (V^{\otimes n})^{h \Sigma_n}. \]
  Moreover the forgetful functor preserves cosifted limits.
\end{lemma}

\begin{proof}
  For the first statement, apply \cite[Proposition 3.1.3.13]{HA} to the opposite forgetful functor
  \[ \cCAlg(\Pro(\scr C))^\op \cong \CAlg(\Pro(\scr C)^\op) \to \Pro(\scr C)^\op. \]  
  The second statement follows from \cite[Proposition 5.4.7.11]{HTT} (applied to the same functor).
\end{proof}

\begin{remark}
  Observe the meaning of the formula for $C(V)$: the infinite product is computed as a formal $\N$-indexed inverse limit of the finite products, which are computed indexwise.
  Similarly $(\ph)^{h\Sigma_n}$ is computed as a formal $\N$-indexed inverse limit of finite limits which are computed indexwise, coming from a filtration of $B\Sigma_n$ by finite skeleta.
\end{remark}

\begin{remark} \label{rmk:pres-cofree}
  One pleasant consequence of Lemma \ref{lemm:cofree-pro} is the following:
  if $\scr C \to \scr D$ is a symmetric monoidal functor preserving finite limits,
  then $\Pro(\scr C) \to \Pro(\scr D)$ preserves cofree coalgebras.
\end{remark}

Still assuming $\scr C$ symmetric monoidal, the constant pro-object functor is also symmetric monoidal.
It hence induces
\[ c: \cCAlg(\scr C) \to \cCAlg(\Pro(\scr C)). \]

\begin{definition}
  We denote its partially defined right adjoint by $M^{\mathrm{cA}}$.
\end{definition}

A comparison of universal properties shows that:

\begin{lemma} \label{ex:M-cofree}
  Assuming that $\scr C$ is presentable with accessible tensor product,
  then $M^{\mathrm{cA}} C(cV)$ exists for any $V \in \scr C$ and is given by $C(V)$.
\end{lemma}

Beware that even if $M^{\mathrm{cA}}(E)$ exists, its underlying object in $\scr C$ need not coincide with $M(E)$.
For example, given a $V \in \scr C$ the cofree coalgebra $C(V)$ is usually \emph{not} given by the expected formula (see Warning \ref{warning:cofree}), and hence $M^{\mathrm{cA}}(CcV) \not\cong M(CcV)$.

\begin{remark} \label{rmk:M-McA-triv}
Let $E \in \cCAlg(\Pro(\scr C))$ such that the underlying object in $\Pro(\scr C)$ is constant, that is, in the essential image of $c$.
In this case it is clear that $E$ lies in the essential image of the fully faithful functor $c: \cCAlg(\scr C) \to \cCAlg(\Pro(\scr C))$ and hence $M^{\mathrm{cA}}(E) \wequi ME$.
\NB{say more?}
\end{remark}

We shall need a slightly more sophisticated criterion to know when $M^{\mathrm{cA}}$ coincides with $M$.
To explain it, we assume that $\scr C$ is stable and provided with a $t$-structure.

\begin{construction}
  Let $\scr C$ be a stable category with a $t$-structure.
  We equip $\Pro(\scr C)$ with the corresponding $t$-structure.
  An object of $\{X_\lambda\} \in \Pro(\scr C)$ is called \emph{pro-truncated} if each $X_\lambda \in \scr C$ is bounded above.
  Denote by $\widehat{\Pro}(\scr C) \subset \Pro(\scr C)$ the full subcategory on pro-truncated objects.

The inclusion of pro-truncated objects admits a left adjoint $\tau_{<\infty}$ called \emph{pro-truncation} which commutes with cofiltered limits
(see \cite[\S4.1.2]{barwick2018exodromy})
\[ \begin{tikzcd}[sep=huge]
  \widehat{\Pro}(\scr C) \ar[r, bend right, hook] & 
  \Pro(\scr C). \ar[l, bend right, "\tau_{<\infty}"'] 
\end{tikzcd} \]
An explicit formula for pro-truncation is given by sending $\{X_\lambda\}_{\lambda \in \Lambda}$ to $\{\tau_{\le n} X_\lambda\}_{(\lambda, n) \in \Lambda \times \N}$.
We say that $E \in \Pro(\scr C)$ is \emph{pro-constant up to pro-truncation} if the canonical map \[ \tau_{<\infty} cM(E) \to \tau_{<\infty} E \] is an equivalence.\NB{in reasonable cases this just means that $\tau_{<n} E$ is constant for all $n$}
\end{construction}

\begin{construction} \label{cnstr:tensor-pro-trunc}
  Let $\scr C$ be a stably symmetric monoidal category equipped with a $t$-structure for which
  a tensor product of bounded above objects is bounded above.
  In this situation the subcategory $\widehat{\Pro}(\scr C)$ is closed under tensor products.
  This allows us to equip $\widehat{\Pro}(\scr C)$ with a symmetric monoidal structure so that the inclusion
  $ \widehat{\Pro}(\scr C) \to \Pro(\scr C) $
  is symmetric monoidal.    
  The left adjoint $\tau_{<\infty}$ is then oplax monoidal.
\end{construction}

\begin{lemma} \label{lem:pro-trunc-sym}
  Let $\scr C$ be a stably symmetric monoidal category equipped with a $t$-structure for which
  \begin{enumerate}
  \item a tensor product of bounded above objects is bounded above,
  \item connectivity is additive under tensor products in $\scr C$ and
  \item the unit of $\scr C$ is bounded.
  \end{enumerate}
  The restriction of the oplax symmetric monoidal functor $\tau_{< \infty}$ from \Cref{cnstr:tensor-pro-trunc}
  to pro-(bounded below) objects is symmetric monoidal.
\end{lemma}

\begin{proof}
  The third assumption ensures that $\tau_{<\infty}$ is strictly unital.
  What remains is to check that maps
  \[ \tau_{< \infty}(X \otimes Y) \to (\tau_{< \infty}X) \otimes (\tau_{<\infty} Y) \]
  provided by the oplax symmetric monoidal structure on $\tau_{<\infty}$
  are isomorphisms for $X,Y$ pro-(bounded below).
  As $\tau_{< \infty}$ and $- \otimes -$ commute with cofiltered limits and every
  pro-(bounded below) $X$ can be written as a cofiltered limit $\varprojlim c X_\lambda$ with the $X_\lambda$ bounded below
  it will suffice to analyze the maps
  \[  \tau_{< \infty}(c A \otimes c B) \to (\tau_{< \infty} c A) \otimes (\tau_{<\infty} c B)  \]
  when $A,B \in \scr C$ are connective.

  Using our assumption that connectivity is additive under tensor products we have that
  \[ \tau_{\leq k}(A \otimes B) \to \tau_{\leq k}( \tau_{\leq n}A \otimes \tau_{\leq m}B ) \]
  is an isomorphism for $k \leq n,m$.
  Taking the limit over $k,n,m \to \infty$ the desired conclusion follows.
\end{proof}

\begin{lemma} \label{lemm:coalg-M}
  Let $\scr C$ be as in \Cref{lem:pro-trunc-sym}.
  Additionally we assume that $\scr C$ has cofiltered limits and the $t$-structure on $\scr C$ is left complete.
  If $E \in \cCAlg(\Pro(\scr C))$ is pro-(bounded below) and pro-constant up to pro-truncation,
  then $M^{\mathrm{cA}}(E)$ exists and coincides with $M(E)$.
\end{lemma}

\begin{proof}
  Let $\widehat{c}$ be the restriction of $\tau_{< \infty} \circ c$ to bounded below objects in $\scr C$.
  The assumptions on $\scr C$ ensure that this functor is fully faithful 
  and provide a symmetric monoidal structure on it (\Cref{lem:pro-trunc-sym}).  
  As a consequence the induced functor
  $ \widehat{c}^{\mathrm{cA}} : \cCAlg(\scr C^{+}) \to \cCAlg(\widehat{\Pro}(\scr C)) $
  is also fully faithful, where $\scr C^+ \subset \scr C$ denotes the subcategory of bounded below objects.

  Together the conditions that $E$ be pro-constant up to pro-truncation and pro-(bounded below) imply that $ME$ is bounded below.
  (The inverse of $\tau_{<\infty} c(ME) \to \tau_{<\infty} E$ yields a map $\tau_{<n} E_\lambda \to \tau_{<0} ME$ for some $n, \lambda \gg 0$ such that the composite $\tau_{<n} ME \to \tau_{<n} E_\lambda \to \tau_{<0} ME$ is the canonical truncation. Since $E_\lambda$ is bounded below, so is $ME$, $\tau_{<0} ME$ being a retract of $\tau_{<0} E_\lambda$.)
  This means that $E$ is in the image of $\widehat{c}^{\mathrm{cA}}$.  
  As $M^{\mathrm{cA}}$ is the (partially defined) right adjoint to
  the fully faithful functor $\widehat{c}^{\mathrm{cA}}$
  and $M$ is the right adjoint to
  the fully faithful functor $\widehat{c}$
  it follows that $M^{\mathrm{cA}}(E)$ exists and its underlying object
  may be indentified with $M(E)$.  
\end{proof}

\begin{corollary} \label{corr:proconst-limits}
  Let $\scr C$ as in Lemma \ref{lemm:coalg-M}.
  Let $E_\bullet \in \cCAlg(\scr C)$ denote a cosifted diagram.
  Assume that the underlying cofiltered diagram in $\scr C$ is pro-constant up to pro-truncation and pro-(bounded below).
  Then the limit of $E_\bullet$ in $\cCAlg(\scr C)$ exists and coincides with the limit in $\scr C$.
\end{corollary}

\begin{proof}
We have $\lim E_\bullet \cong M^{\mathrm{cA}} \lim cE_\bullet$.
The latter limit may be computed on underlying pro-objects (see Lemma \ref{lemm:cofree-pro}), and hence by assumption is pro-constant up to pro-truncation and pro-(bounded below).
Consequently we may apply \Cref{lemm:coalg-M} and see that $UM^{\mathrm{cA}} \lim cE_\bullet \wequi M U\lim cE_\bullet \wequi \lim UE_\bullet$, as desired.
\end{proof}

The following criterion helps us to detect objects which are pro-constant up to pro-truncation.
\begin{lemma} \label{lemm:detect-constant}
  Let $\scr C = \Mod_k$ where $k$ is a field.
  Let $\{X_\lambda\} \in \Pro(\Mod_k)$ be a pro-(bounded below) object such that
  each $\{\pi_i X_\lambda\}_{\lambda} \in \Pro(\Mod_k^\heart)$ is pro-constant, and is isomorphic to $0$ for $i<0$.
  Then $\{X_\lambda\}$ is connective and pro-constant up to pro-truncation.
\end{lemma}

\begin{proof}
  Replacing $X$ by $\tau_{<\infty} X$, we may assume that each $X_\lambda$ is bounded.
  
  First assume that all the homotopy pro-objects are $0$.
  We shall show $X=0$.
  To do this, let $\lambda \in \Lambda$.
  Since $\{\pi_i X_\mu\}_\mu$ is zero, there exists $\mu(i) > \lambda$ such that $\pi_i(X_{\mu(i)}) \to \pi_i(X_{\lambda})$ is the zero map.
  Since $X_\lambda$ is bounded, we may choose $\mu(i) = \mu$ independently of $i$.
  Now the map $X_\mu \to X_\lambda$ is null, since we are working with $k$-modules (which split into their homotopy objects).
  
  Now we treat the general case.
  We have the fiber sequence $\tau_{\geq 0}X \to X \to \tau_{<0}X$.
  The previous case applies to $\tau_{<0}X$, whence $\tau_{\geq 0}X \cong X$.
  Set $Y = \lim X$.
  Since each $\pi_i X$ is pro-constant, we see that $\pi_i Y \cong \lim \pi_i X$.
  From this it follows that $Y_{\le n} \to X_{\le n}$ is an isomorphism (being obtained as an extension of finitely many isomorphisms, one for each of the homotopy objects).
  This was to be shown.
\end{proof}

Finally, we record a criterion for recognizing when pro-objects in an abelian category are pro-constant.

\begin{lemma}\label{lemm:constancy}
  Let $\scr C$ be a artinian abelian category with enough projectives.
  Every countable pro-object $\{ F_{\lambda} \} \in \Pro_\omega(\scr C)$
  which admits a limit in $\scr C$ is pro-constant.
\end{lemma}

\begin{proof}
  Recall that $\Pro(\scr C)$ is abelian \cite[Theorem 4.1(iv)]{staufer-completion}; in particular a map is an isomorphism if and only if it has vanishing kernel and cokernel (which may be computed indexwise).
  Let $\tilde F$ be the limit of $F$ viewed as a constant pro-object. We will consider the exact sequence of pro-objects
  \[ 0 \to K \to \tilde F \to F \to C \to 0 \]
  and we shall prove that $K = 0 = C$.
  As $K$ is a subobject of $\tilde F$ which is constant all of its transition maps are injective. As $\tilde F$ is artinian, this implies that $K$ is pro-constant and therefore since $\tilde F$ is the limit of $F$ we must have that $K$ is pro-isomorphic to $0$. Again using the fact that $\scr C$ is artinian we may form the pointwise sub-pro-object of eventual minimums of $F$ i.e. the pro-object $F'$ with $F'_\lambda = \cap_{\lambda' \to \lambda} \img(F(\lambda') \to F(\lambda))$. The quotient $\coker(F' \to F)$ is pro-isomorphic to zero (again using artinianness). Now we may lift the map $\tilde F \to F$ to a map $\tilde F \to F'$. Let $P$ be a projective object of $\scr C$. Then $\Map(P, F'_\lambda)$ is a countable pro-system of abelian groups along epimorphisms, and so it follows that $\Map(P, \tilde F) = \lim_\lambda \Map(P, F'_\lambda) \to \Map(P, F'_\lambda)$ is surjective for all $\lambda$. Since $\scr C$ has enough projectives and $P$ was arbitrary, this implies that $\tilde F \to F'_\lambda$ is epi for all $\lambda$. This verifies that we have isomorphisms of pro-objects $C \cong \coker(\tilde F \to F) \cong \coker(\tilde F \to F') \cong 0.$ This concludes the proof.
\end{proof}

\section{Idempotents in coalgebras}
\label{sec:idem}

In this section we establish some results about inverting idempotent elements in coalgebras.
These results will be used to reduce the proof of our main theorems to the connected case.

Throughout this section we work in an idempotent complete, stably symmetric monoidal category $\scr C$.

\begin{definition}
Let $A \in \cCAlg(\scr C)$.
By an \emph{element of $A$} we mean a homotopy class of maps $A \to \1$ in $\scr C$.
Using the comultiplication on $A$, we can multiply elements of $A$.
Similarly for any element $a$ of $A$ we obtain an endomorphism of the underlying object $A \in \scr C$ ``multiplication by $a$''.
We call an element \emph{idempotent} if it satisfies $e^2 = e$.
We denote the element specified by the co-unit of $A$ by $1=1_A$.
We call an element $x$ a \emph{unit} if there exists an element $y$ with $xy = 1$.
\end{definition}

\begin{remark}
It may be helpful to pass to $\scr C^\op$, and use that $\cCAlg(\scr C)^\op = \CAlg(\scr C^\op)$.
We can thus view $A \in \cCAlg(\scr C)$ as $A \in \CAlg(\scr C^\op)$, and elements of $A$ are just maps $\1 \to A \in \scr C^\op$.
\end{remark}

\begin{remark}
Note that the set of elements is contravariantly functorial: given $f: A \to B \in \cCAlg(\scr C)$ and an element $x$ of $B$, there is a canonical element $f^*(x) = x \circ f$ of $A$.
\end{remark}

\begin{example} \label{ex:tautological-idempotents}
Let $A, B \in \cCAlg(\scr C)$.
The coproduct $A \amalg B \in \cCAlg(\scr C)$ exists in $\cCAlg(\scr C)$; the underlying object is $A \oplus B$.
The composite $A \oplus B \xrightarrow{pr} A \xrightarrow{c_A} \1$, where $c_A$ is the co-unit of $A$, defines an element $e_A$ of $A \amalg B$ which is easily verified to be idempotent.
We similarly construct an idempotent $e_B$, and observe that $1_{A \amalg B} = e_A + e_B$.
\end{example}

\begin{definition}
Let $A \in \cCAlg(\scr C)$ and $a \in A$.
Let $\scr D \subset \cCAlg(\scr C)_{/A}$ denote the full subcategory on those objects $(f: B \to A)$ such that $f^*(a)$ is a unit of $B$.
We denote by $A[a^{-1}]$ the terminal object of $\scr D$, if it exists.
\end{definition}

\begin{remark} \label{rmk:def-localization-coalg}
In other words for $T \in \cCAlg(\scr C)$, $\Map(T, A[a^{-1}]) \to \Map(T, A)$ is a monomorphism onto those connected components corresponding to maps $f: T \to A$ with $f^*(a)$ a unit.
\end{remark}

Because of the usual difficulties with coalgebras, it is not clear if localizations $A[a^{-1}]$ exist. They do, however, exist whenever $a$ is idempotent.

\begin{lemma} \label{lemm:coalg-idempotents}
  Let $A \in \cCAlg(\scr C)$ and $e$ an idempotent element of $A$.
  Then $A[e^{-1}]$ exists and its underlying object is
  given by inverting the idempotent map $e : A \to A$ in $\scr C$.
\end{lemma}

\begin{proof}
  We work with pro-objects.
  Comparing the universal properties, it is clear that if $(cA)[e^{-1}]$ exists then $M^{\mathrm{cA}}((cA)[e^{-1}]) \wequi A[e^{-1}]$; in particular the latter exists.
  Dualizing the standard theory for inverting elements in $\E_\infty$-algebras (see e.g. \cite[Lemma 12.1]{norms}), we see that for any $B \in \cCAlg(\Pro(\scr C))$ and element $b \in B$, $B[b^{-1}]$ exists and has underlying object given by \[ \lim(B \xleftarrow{\times b} B \xleftarrow{\times b} B \xleftarrow{\times b} \dots ). \]
  Let $A \wequi A_1 \oplus A_2$ be the idempotent decomposition of (the underlying object of) $A$, with $A_1$ corresponding to $e$ and $A_2$ corresponding to $(1-e)$.
  Then $cA \wequi cA_1 \oplus cA_2$ (the functor $c$ preserves colimits), which implies that $(cA)[e^{-1}] \wequi cA_1$ is constant.
  It follows that $A[e^{-1}]$ exists and is given by $M^{\mathrm{cA}}((cA)[e^{-1}]) \wequi M^{\mathrm{cA}}(cA_1) \wequi A_1$ (see Remark \ref{rmk:M-McA-triv}), as claimed.
\end{proof}

\begin{remark} \label{rmk:summands-mono-coalg}
  This result shows in particular that in $\cCAlg(\scr C)$, the summand inclusions $A_1 \to A_1 \amalg A_2$ are monomorphisms (because they coincide with localizations, which are monomorphisms by definition, see Remark \ref{rmk:def-localization-coalg}).
\end{remark}

\begin{proposition} \label{prop:coprod-coalg}
  Assume that the commutative ring $[\1, \1]_{\scr C}$ is non-zero and has no non-trivial idempotents.
  The functor $\Map_{\cCAlg(\scr C)}(\1, \ph)$ preserves finite coproducts.
  If in addition we assume that $\1 \in \scr C$ is compact and $\scr C$ has arbitrary coproducts, then the functor preserves all coproducts.
\end{proposition}

\begin{proof}
  We begin with the case of empty coproducts.
  If there is a map of coalgebras $\1 \to 0$, then composing with the co-unit exhibits $\1$ as a retract of $0$.
  This contradicts our assumption that $[\1, \1]_{\scr C}$ is non-zero.

  Next we treat the case of binary coproducts.
  Let $A_1, A_2 \in \cCAlg(\scr C)$ and set $A = A_1 \amalg A_2$.
  Recall from Example \ref{ex:tautological-idempotents} the two tautological idempotents $e_1, e_2$ of $A$.
  They have the property that $e_1 + e_2 = 1$, $e_1 e_2 = 0$ and $A[e_i^{-1}] \wequi A_i$ (Lemma \ref{lemm:coalg-idempotents}).
  Let $f: \1 \to A \in \cCAlg(\scr C)$ be any morphism.
  Then $f^*(e_1), f^*(e_2)$ are two orthogonal idempotents in $\1$, adding up to $1$.
  By assumption this implies that $\{f^*(e_1), f^*(e_2)\} = \{1,0\}$.
  In particular precisely one of $f^*(e_i)$ is a unit.
  This implies that the sum of monomorphisms (see Remark \ref{rmk:summands-mono-coalg}) \[ \Map(\1, A_1) \amalg \Map(\1, A_2) \to \Map(\1, A) \] is an equivalence (i.e. a bijection on $\pi_0$).

  Finally we treat arbitrary coproducts (assuming they exist and $\one \in \scr C$ is compact). 
  Let $A_i \in \cCAlg(\scr C)$ for $i \in I$.
  Then $ A = \coprod_{i \in I} A_i \in \cCAlg(\scr C) $ exists and has underlying object $\bigoplus_i A_i$.
  For each $i$, we obtain an idempotent $e_i$ of $A$, corresponding to the splitting of the underlying object.
  Consider now a morphism $f: \1 \to A \in \cCAlg(\scr C)$.
  It will suffice to show that $f^*(e_i)$ is invertible (and in fact equal to $1$) for precisely one $i$; indeed this implies that the sum of monomorphisms (see Remark \ref{rmk:summands-mono-coalg}) \[ \coprod_{i \in I} \Map(\1, A_i) \to \Map(\1, A) \] is an equivalence.
  To prove the claim, note that by compactness of $\1 \in \scr C$, the underlying map $\1 \to \bigoplus A_i$ factors through a finite sum $\bigoplus_{i \in S} A_i$.
  This implies that $f^*(e_i) = 0$ for $i \not\in S$ and also that $\sum_{i \in S} f^*(e_i) = 1$ (since $f$ preserves the co-units), from which it follows as in the finite case that $f^*(e_i) = 1$ for precisely one $i \in S$, and $f^*(e_i) = 0$ for all other $i$.
\end{proof}

\section{The Artin--Schreier pullback square}
\label{sec:key-lemma}

In this section we shall prove that the commutative square \eqref{eq:key-pullback} is a pullback.
To do this, we shall employ three techniques: using pro-objects, treating all vector spaces together at the same time as a single universal example and exploiting polynomiality to control the universal example.



\begin{definition}
  Throughout this section we will let $\scr C \coloneqq \Fun^{\kappa}(\Set_*^\ainj, \Mod_k)$
  denote the category of $\kappa$-accessible functor $\Set_*^{\ainj} \to \Mod_k$.  
\end{definition}

\begin{lemma} \label{lem:cofree-commute}
  The category $\scr C$ is presentable.
  Evaluation at a set $X$ provides a symmetric monoidal left adjoint
  \[ \ev_X : \scr C \to \Mod_k \]
  and the cofree coalgebra functors commute with $\ev_X$.    
\end{lemma}

\begin{proof}
  From \cite[Propositions 5.4.2.9 and A.2.6.4]{HTT} we know that $\Set_*^{\ainj}$ is $\kappa$-accessible.
  We may therefore identify $\scr C$ with $\Fun((\Set_*^\ainj)^{\kappa}, \Mod_k)$.
  As the source category is now small, we see that $\scr C$ is presentable.
  The symmetric monoidal structure on this functor category is pointwise, so $\ev_X$ is visibly symmetric monoidal.
  Since $\cCAlg(\Fun((\Set_*^\ainj)^{\kappa}, \Mod_k)) \wequi \Fun((\Set_*^\ainj)^{\kappa}, \cCAlg(\Mod_k))$ \cite[Remark 2.1.3.4]{HA} we see that cofree coalgebras are computed pointwise,
  so we learn the cofree coalgebra functors commutes with $\ev_X$ for any $\kappa$-small $X$.
  On the other hand, since $C_k$ is $\kappa$-accessible this property extends to evaluation at $X$ for all $X$.
\end{proof}




\begin{definition}
  An object $X \in \scr C$ is \emph{polynomial} if
  the induced functors $\pi_iX : \Set_*^\ainj \to \Mod_k^\heartsuit$ 
  are polynomial in the sense of \Cref{dfn:polynomial}
  and vanish for all but finitely many $i \in \Z$.
  We write $\scr C^{\poly} \subseteq \scr C$ for the full subcategory of those $X \in \scr C$ which are polynomial.  
\end{definition}

As a consequence of the closure properties for polynomial functors from \Cref{lem:poly-closure} we note that $\scr C^{\poly}$ is closed under finite (co)limits, retracts and tensor product.



For our purposes the most important object in $\scr C$ will be the $\scr V \in \scr C^{\heartsuit}$ introduced in Example \ref{ex:poly-free}.
$\scr V$ is constructed to be the ``universal vector space equipped with a basis''
and we have $ \ev_X(\Sigma^n \scr V) \cong \Sigma^n k\{X\} $, for any $X \in \Set_*^{\ainj}$.

Now we turn to the categories of pro-objects we will need.
As the inclusion $\scr C^{\poly} \hookrightarrow \scr C$ is symmetric monoidal and exact the induced functor
\[ \Pro( \scr C^{\poly} ) \to \Pro(\scr C) \]
is also fully faithful, symmetric monoidal, and preserves limits.
We say that an $X \in \Pro(\scr C)$ is polynomial if it is in the image of this functor.

Our first task is now to lift the Frobenius functional on $\Sigma^n \scr V$ constructed in \Cref{sec:frob} to a Frobenius functional at the level of pro-objects.

\begin{lemma} \label{lem:pro-frob}
  Let $n \ge 0$. There is a map
  \[ \varpi : (c(\Sigma^n \scr V)^{\otimes p})^{hC_p} \to c(\Sigma^n \scr V) \]
  in $\Pro(\scr C)$ such that the map $M(\varpi)$ coincides with
  the map obtained by restricting the projection map from \Cref{def:artin-schreier} along $X \mapsto S^n \wedge X$.
\end{lemma}

\begin{proof}

  As $M$ is a right adjoint, there is an assembly map
  \[ (M(-))^{tC_p} \to M((-)^{tC_p}). \]
  Our first observation is that as a consequence of the fact that $c(-)$ is a left adjoint and $M(c(-)) \cong \Id$,
  this assembly map becomes an isomorphism after precomposition with $c(-)$.
  As the projection map from \Cref{def:artin-schreier} factored through $T(\Sigma^n \scr V)$
  it will now suffice to argue that materialization induces a surjective map 
  \[ \pi_0\Map_{\Pro(\scr C)}( T(c\Sigma^n \scr V), c \Sigma^n \scr V ) \to \pi_0\Map_{\scr C}( M T(c\Sigma^n \scr V), M c \Sigma^n \scr V ). \]

  As the cofiber of the map $a_{n\rho} : \Ss^n \to (\Ss^n)^{\otimes p}$ is built out of finitely many induced cells (note that the finite pointed $C_p$-CW complex $(S^n)^{\wedge p}/S^n$ has free action),
  the map $\Sigma^n c \scr V^{\otimes p} \to c(\Sigma^n \scr V)^{\otimes p}$ becomes an isomorphism upon applying $(-)^{tC_p}$ \cite[Lemma I.3.8(i)]{NS}.
  This allows us to reduce to the case $n=0$.

  The object $\scr V^{\otimes p}$ is given by the functor $S \mapsto k\{S^{\wedge p} \}$.
  The diagonal then defines a map $ \Delta : \scr V \to \scr V^{\otimes p} $.
  We further write $\scr V' \subseteq \scr V^{\otimes p}$ for the subfunctor $S \mapsto k\{S^{\wedge p} \setminus \Delta (S) \}$.\footnote{Note that the functor $\scr V' : \Set_*^\ainj \to \Mod_k^\heartsuit$ does not extend to a functor on $\Set_*$. This is the origin of our restriction to almost-injective maps in previous sections.}  
  Together these maps give us a splitting $\scr V^{\otimes p} \cong \scr V \oplus \scr V'$.
  Upon materializing and evaluating at a finite set $X$ we note that the $C_p$ action on $\scr V'$ is induced and therefore that the projection map
  $r : \scr V^{\otimes p} \to \scr V$
  is sent to an isomorphism by $M((-)^{tC_p})$.
  The upshot of this is that it will now suffice to instead prove that materialization induces a surjection
  \[ \pi_0\Map_{\Pro(\scr C)}( (c \scr V)^{tC_p} , c \scr V ) \to \pi_0\Map_{\scr C}( M (  (c\scr V)^{tC_p} ), M c \scr V ). \]
  
  Applying \Cref{lem:tate-trivial} (note that $c \scr V$ has trivial $C_p$-action) we obtain an isomorphism
  \[ (c \scr V)^{tC_p} \cong \left( \prod_{0 > j } \Sigma^j c \scr V \right) \oplus c \scr V \oplus \left( \bigoplus_{j > 0} \Sigma^j c \scr V \right). \]
  Using the fact that $c$ commutes with colimits and the cocompactness of objects in the image of $c$ we learn that the map above is equivalent to the map 
  \[ \pi_0\left( \colim_{i \to -\infty} \prod_{j \geq i} \Map_{\scr C}( \Sigma^j \scr V, \scr V ) \right) \to \pi_0\left( \prod_{j} \Map_{\scr C}( \Sigma^j \scr V, \scr V ) \right). \]
  Taking $\pi_0$ commutes with products and filtered colimits so what remains is to compute $\pi_0$ of the individual mapping spaces above.
  To complete the proof, we will show that $\pi_0\Map_{\scr C}( \Sigma^j \scr V, \scr V )$ vanishes for $j \ne 0$.
  To see this, note that $\scr V$ is a retract of the $k$-linearization of the corepresentable functor $\Map_{\Set_*^\ainj}(S^0, -)$.
  We hence have
  \[ \pi_0\Map_{\scr C}( \Sigma^j \scr V, \scr V ) \subset \pi_0( \Sigma^{-j} \ev_{S^0} \scr V) \cong \begin{cases} k, & j=0 \\ 0, & j \neq 0 \end{cases}. \]
\end{proof}

\begin{construction} \label{cnstr:pro-AS}
  Following the recipe from \Cref{def:artin-schreier} we construct a \emph{Frobenius functional} $F$ for $c\Sigma^n \scr V$
  as the composite of the canonical arity $p$ co-operation
  \[  C(c \Sigma^n \scr V) \xrightarrow{\Delta} (C(c \Sigma^n \scr V)^{\otimes p})^{hC_p} \xrightarrow{c} ((c \Sigma^n \scr V)^{\otimes p})^{hC_p} \]
  with a choice of projection 
  \[ ((c \Sigma^n \scr V)^{\otimes p})^{hC_p} \xrightarrow{\varpi} c \Sigma^n \scr V \]
  as provided by \Cref{lem:pro-frob}.
  Passing to the mate of $\epsilon - F$ (recall that $\epsilon$ is the counit of the forgetful--cofree adjunction)
  we construct the \emph{pro-(Artin--Schreier map)}
  \[ 1-F \colon C(c \Sigma^n \scr V) \to C(c \Sigma^n \scr V). \]
  For $n \ge 0$, let $\scr Q_n$ be defined by the pullback square in $\cCAlg(\Pro(\scr C))$
  \begin{equation}
    \begin{CD}
      \scr Q_n @>>> C(c\Sigma^n \scr V) \\
      @VVV              @V{1-F}VV \\
      ck  @>>> C(c\Sigma^n \scr V).
    \end{CD} \label{second-square}
  \end{equation}
\end{construction}

\begin{lemma} \label{lemm:pro-coalg-frob}
  Let $n \ge 0$.
  The image under $M^{\mathrm{cA}}$ of the cospan from \eqref{second-square} is the restriction of the cospan from \eqref{eq:key-pullback} to $\Set_*^{\ainj}$.
\end{lemma}

\begin{proof}
  Example \ref{ex:M-cofree} provides an isomorphism $M^{\mathrm{cA}}Cc(-) \cong C(-)$ and
  cofree coalgebras in $\scr C$ are computed pointwise (\Cref{lem:cofree-commute}) so the objects are the same.
  The maps between them are identified using the universal properties of cofree coalgebras and
  the compatibility with materialization from \Cref{cnstr:pro-AS}.
\end{proof}

\begin{corollary} \label{cor:pullback}
Let $W \in \Set_*^\ainj$, $n \ge 0$.
Then $M^{\mathrm{cA}}\ev_W \scr Q_n$ is the pullback of the cospan in \eqref{eq:key-pullback}.
\end{corollary}

\begin{proof}
  As $M^{\mathrm{cA}}$ preserves limits, this follows from \Cref{lemm:pro-coalg-frob}.
\end{proof}

The following is essentially a reformulation of Mandell's theorem \cite{mandell-e-infty}.
\begin{lemma} \label{lemm:mandell}
  Let $n \ge 1$.\NB{$n \ge 1$ because that's what Mandell assumes. Lurie does not, as far as I can tell...}\NB{Lurie also just refers to Mandell without saying anything about reducing $k=\overline{F}_p$...}
  $\ev_{S^0} \scr Q_n$ is connective and pro-constant up to pro-truncation with with finite dimensional homotopy groups.
  Moreover, the square \eqref{eq:key-pullback} is a pullback for $W=S^0$.
\end{lemma}

\begin{proof}
The object $\ev_{S^0} \scr Q_n$ lies in $\Pro(\Mod_k)$ and is obtained as a pullback involving $ck$ and $C(c\Sigma^n k)$.
Since the ingredients lie in $\Pro(\Mod_k^\omega)$, we deduce (using Remark \ref{rmk:pres-cofree}) that $\ev_{S^0} \scr Q_n$ lies in $\Pro(\Mod_k^\omega)$ as well.
Observe that taking duals induces a $t$-exact, symmetric monoidal equivalence $(\Mod_k^\omega)^\op \cong \Mod_k^\omega$, and hence also \[ \Pro(\Mod_k^\omega)^\op \cong \Ind(\Mod_k^\omega) \cong \Mod_k \qquad\text{and}\qquad \cCAlg(\Pro(\Mod_k^\omega))^\op \cong \CAlg(\Mod_k). \]

We may therefore think of $\ev_{S^0} \scr Q_n$ as an ordinary $\E_\infty$-ring in $\Mod_k$.
In fact this equivalence interchanges cofree and free objects and respects the Artin--Schreier map construction (See \Cref{rmk:Frob-Q0}).
Thus by Mandell's theorem \cite[Theorem 6.2 and proof of Proposition A.7]{mandell-e-infty}, $\ev_{S^0} \scr Q_n$ corresponds to the $\E_\infty$-$k$-algebra $k^{K(\F_p, n)}$.
In particular the homotopy groups are degreewise finite \cite[Theorem 4]{Cartan-EM}.
In other words, the homotopy groups viewed as objects in $\Ind(\Mod_k^{\omega,\heart})$ have colimit again in $\Mod_k^{\omega,\heart}$.
Lemma \ref{lemm:constancy} then implies that they must be constant as ind-objects.\footnote{This conclusion follows already from the equivalence $\Ind(\Mod_k^{\omega,\heart}) \cong \Mod_k^\heart$, but we shall re-use the argument later in a situation where this trick is not available.}
This proves that $\ev_{S^0} \scr Q_n$, now viewed again as a pro-object, has pro-constant homotopy groups which vanish for $i < 0$.
In particular, by Lemma \ref{lemm:detect-constant}, it is pro-constant up to pro-truncation, and hence by Lemma \ref{lemm:coalg-M} we have $M^{\mathrm{cA}} \ev_{S^0} \scr Q_n \cong M \ev_{S^0} \scr Q_n$. 

Knowing now that $\ev_{S^0} \scr Q_n$ is pro-constant up to pro-truncation it will suffice to instead show that the map
$ \tau_{< \infty} c k[K(\F_p, n)] \to \ev_{S^0} \scr Q_n $ is an isomorphism or equivalently that the square
\[ \begin{tikzcd}
  \mathrm{Free}_k^{\E_\infty}(\Sigma^{-n}k) \ar[r, "1-F"] \ar[d] &
  \mathrm{Free}_k^{\E_\infty}(\Sigma^{-n}k) \ar[d] \\
  k \ar[r] &
  k^{K(\F_p,n)}
\end{tikzcd} \]
is pushout square of augmented $\E_\infty$-$k$-algebras\footnote{Recall that the forgetful functor from augmented algebras to algebras preserves pushouts.}.
The key point here is that the right vertical map sends the class in degree $-n$ to the fundamental class
(indeed this is how this map was constructed in \Cref{cnstr:AS-square}).
As there exists a pushout square with the same vertices and maps by \cite[Theorem 6.2 and proof of Proposition A.7]{mandell-e-infty} what remains is to check the $2$-cells agree.
On the other hand the space of maps of augmented $\E_\infty$-$k$-algebras $\mathrm{Free}_k^{\E_\infty}(\Sigma^{-n}k) \to k^{K(\F_p,n)}$ is $0$-truncated so the space of choices of the $2$-cell is contractible (and therefore they agree).
This concludes the proof.
\end{proof}

\begin{lemma} \label{lem:finite-case}
  Given $V, W \in \Set_*^\ainj$ the projection maps $V \vee W \to V$ and $V \vee W \to W$ (collapsing the other set to the base point) induce isomorphisms of coalgebras
  \[ \ev_{V \vee W} \scr Q_n \cong \ev_V \scr Q_n \otimes \ev_W \scr Q_n. \]  
  Consequently:
  \begin{enumerate}[(i)]
  \item For $V$ finite, $\ev_V \scr Q_n \in \Pro(\Mod_k)$ is connective and pro-constant up to pro-truncation with finite dimensional homotopy groups.
  \item For $V$ finite, the square \eqref{eq:key-pullback} is a pullback square.
  \end{enumerate}
\end{lemma}

\begin{proof}  
  Let $V, W \in \Set_*^\ainj$.
  We have the ``projection maps'' $V \vee W \to V$ and $V \vee W \to W$, collapsing the other set to the base point.  
  Upon evaluating $\scr V$ at these maps we obtain actual projection maps for a product diagram in $\Mod_k$.
  As $C$ and $c$ both preserve finite products we learn that the square \eqref{second-square} sends these maps to a product diagram.\footnote{Recall that the product of coalgebras is given by the tensor product.}

  We now prove (i) and (ii) by induction on $|V|$.
  The base case $V \cong S^0$  is \Cref{lemm:mandell}.
  Now write $V \cong W \vee S^0$.
  For (i) we observe that the collection of objects which are connective and pro-constant up to pro-truncation with finite dimensional homotopy groups
  is closed under tensor products. The main part of the lemma and our inductive hypothesis now lets us conclude $\ev_V \scr Q_n$ has these properties as well.
  For (ii) we observe that the map
  \[ k[\Omega^\infty\Sigma^n \F_p\{W \vee S^0\}] \to k[\Omega^\infty\Sigma^n \F_p\{ W\}] \otimes k[\Omega^\infty\Sigma^n \F_p\{ S^0 \}] \]
  is also an isomorphism. Using \Cref{cor:pullback} together with the main part of the lemma, our inductive hypothesis and the fact that $M^{\mathrm{cA}}$ preserves products we conclude that the square is pullback for $V$ as well.
\end{proof}

Now we prove the main result of this section.

\begin{proposition} \label{prop:key}
  For every $W \in \Set_*^\ainj$ and $n \ge 1$ the square \eqref{eq:key-pullback} is a pullback square.
\end{proposition}

\begin{proof}
  We begin by observing that because $c \scr V \in \Pro_{\omega}(\scr P^{\poly})$
  and this subcategory is closed under countable limits and tensor products, \Cref{lemm:cofree-pro} implies that
  $C(c \scr V)$ lies in $\Pro_{\omega}(\scr P^{\poly})$.
  Similarly, expanding $\scr Q_n$ as a totalization of a diagram of products (tensor products on underlying) of copies of $C(c \scr V)$ and $ck$
  we see that $\scr Q_n$ lies in $\Pro_{\omega}(\scr P^{\poly})$ as well.

  The next step is analyzing the homotopy groups $\pi_i\scr Q_n \in \Pro (\scr C^{\poly})^{\heart}$.
  We begin by analyzing the image under the functor
  \[ \ev : \Pro (\scr C^{\poly})^{\heart} \to \Fun(\Fin_*^\ainj, \Pro \Mod_k^\heart) \]
  Using \Cref{lem:finite-case} we conclude that
  $ \ev (\pi_i \scr Q_n) \in \Fun(\Fin_*^\ainj, \Pro \Mod_k^\heart) $
  is zero for $i < 0$,
  pro-constant with value $k$ for $i=0$
  and pro-constant satisfying
  \[ \ev (\pi_i \scr Q_n)(S^0 \vee W) \cong \ev (\pi_i \scr Q_n)(W) \oplus \bigoplus_{j=1}^i \ev (\pi_j \scr Q_n)(S^0) \otimes \ev (\pi_{i-j} \scr Q_n)(W) \]
  for $i > 0$.
  In particular, proceeding by induction on $i$ and using the fact that a functor $F$ is polynomial if $D(F)$ is polynomial
  we see that $\ev (\pi_i \scr Q_n)$ is pro-constant with value a polynomial functor.

  As the (finite) evaluation functor
  \[ (\scr C^{\poly})^{\heart} = \PolyFun \to \Fun(\Fin_*^\ainj, \Mod_k^\heart) \]
  is fully faithful, it reflects limits.
  We may then interpret the above as saying that each of the pro-objects $\pi_i \scr Q_n$ admits a limit in $(\scr C^{\poly})^{\heart}$.
  Applying Lemma \ref{lemm:constancy} using Corollary \ref{cor:polyfun-artinian}, Example \ref{ex:proj-gens} and the fact that these are countable pro-objects,
  we learn that $\pi_i \scr Q_n$ is pro-constant, polynomial and vanishes for $i<0$.

  For $W \in \Set_*^\ainj$, we know that $\ev_W \scr Q_n$ is pro-(bounded below).
  As $\pi_i \scr Q_n$ is pro-constant and vanishes for $i<0$, the homotopy groups
  $\pi_i \ev_W \scr Q_n \cong \ev_W \pi_i \scr Q_n$ share these properties.
  Applying \Cref{lemm:detect-constant} we conclude that $\ev_W \scr Q_n$ is pro-constant up to pro-truncation.  
  Using Lemma \ref{lemm:coalg-M} we learn that $M^{\mathrm{cA}} \ev_W \scr Q_n  \cong M \ev_W \scr Q_n$.
  The homotopy groups of this object are given by evaluating the polynomial functor $M (\pi_* \scr Q_n)$ at $W$.
  In particular they are compatible with filtered colimits in $W$.
  Since the same is true for the left hand side of \eqref{eq:key-pullback}, the proof is reduced to the case where $W$ is finite
  (where the conclusion was proved in \Cref{lem:finite-case}).
\end{proof}

We conclude the section by using the Artin--Schreier pullback square to analyze coalgebra maps into chains on an Eilenberg--MacLane space.

\begin{lemma} \label{lemm:spacelike-homology-maps}
  Let $k$ be separably closed, let $E \in \cCAlg_k$ and let $V$ be an $\F_p$-vector space.
  If $\pi_i(E)$ and $\pi_{i-1}(E)$ are \solvable, then the canonical map
  \[ \pi_i : \pi_0\Map_{\cCAlg_k}(E, k[K(V, i)]) \to \Hom_{\VectFk}(\pi_i(E), k \otimes V ) \]
  is an isomorphism.
\end{lemma}

\begin{proof}
  Proposition \ref{prop:key} supplies us with a pullback square
  \begin{equation*}
    \begin{CD}
      \Map_{\cCAlg_k}(E, k[K(V, i)]) @>>> \Map_{\Mod_k}(E, \Sigma^i k \otimes V) \\
      @VVV  @V{1-F}VV \\
      * @>>> \Map_{\Mod_k}(E, \Sigma^i k \otimes V)
    \end{CD}
  \end{equation*}
  and hence an exact sequence
  \begin{gather*} [\Sigma E, \Sigma^i k \otimes_{\F_p} V]_{\Mod_k} \xrightarrow{F-1} [\Sigma E, \Sigma^i k \otimes_{\F_p} V]_{\Mod_k} \to \\ \to \pi_0 \Map_{\cCAlg_k}(E, k[K(V, i)]) \to [E, \Sigma^i k \otimes_{\F_p} V]_{\Mod_k} \xrightarrow{F-1} [E, \Sigma^i k \otimes_{\F_p} V]_{\Mod_k}. \end{gather*}
  Rewriting this exact sequence using the isomorphisms
  $[E, \Sigma^i k \otimes V]_{\Mod_k} \wequi \Hom(\pi_i E, k \otimes V)$ and $[\Sigma E, \Sigma^i k \otimes V]_{\Mod_k} \wequi \Hom(\pi_{i-1} E, k \otimes V)$
  and identification of the map $1-F$ from \Cref{rmk:F-vs-F} we obtain short exact sequences
  \[ 0 \to \mathrm{Ext}^1_{\VectFk}(\pi_{i-1}E, k \otimes V) \to \pi_0 \Map_{\cCAlg_k}(E, k[K(V, i)]) \to \mathrm{Ext}^0_{\VectFk}(\pi_iE, k \otimes V) \to 0. \]
  The assumption that $\pi_{i-1}E$ and $\pi_iE$ are \solvable, together with the fully faithfulness from \Cref{lem:solvable-ff} now complete the proof.
\end{proof}

\section{Main result} \label{sec:main}

Denote by $\Spc_p^{\mathrm{nil}}$ the category of nilpotent (which for us does not require connected) \cite[Chapter II, \S4.3]{bousfield1987homotopy}, $p$-complete (i.e., local for the functor $\F_p[-]$) spaces.
We view $\Spc_p^{\mathrm{nil}}$ as a cartesian symmetric monoidal $\infty$-category.
By the Künneth theorem, the functor
\[ k[-]: \Spc_p^{\mathrm{nil}} \to \Mod_k \]
is symmetric monoidal and hence induces
\[ k[-]: \cCAlg(\Spc_p^{\mathrm{nil}}) \to \cCAlg(\Mod_k). \]
Since $\Spc_p^{\mathrm{nil}}$ is \emph{cartesian} symmetric monoidal, the forgetful functor
$\cCAlg(\Spc_p^{\mathrm{nil}}) \to \Spc_p^{\mathrm{nil}}$
is an equivalence \cite[Corollary 2.4.3.10]{HA}.
The following is our main result.

\begin{theorem} \label{thm:main}
  Let $k$ be a separably closed field of characteristic $p$.
  The canonical functor
  \[ k[-] \colon \Spc_p^{\mathrm{nil}} \to \cCAlg_k \]
  is fully faithful.
\end{theorem}

\begin{proof}
  The functor $k[-]$ has a right adjoint, given by $\Map_{\cCAlg_k}(k, -)$.
  We must show that the unit transformation
  \[ X \to \Map_{\cCAlg_k}(k, k[X]) \]
  is an equivalence for all $X \in \Spc_p^{\mathrm{nil}}$.
  Let $\DD \subseteq \Spc_p^{\mathrm{nil}}$ be the full subcategory of $p$-nilpotent spaces
  for which this unit transformation is an equivalence.
  Before proceeding we highlight some closure properties of $\DD$:
  \begin{enumerate}
  \item[(a)] $\DD$ is closed under finite products.
  \item[(b)] If $X^\bullet$ is a cosifted diagram with each $X^\alpha \in \DD$ such that
    the underlying pro-object of the $k[X^\bullet]$ is pro-constant up to pro-truncation
    and $U(k[ \lim X^\bullet ]) \cong \lim U(k[X^\bullet])$, then $\lim X^\bullet \in \DD$.
  \item[(c)] $\DD$ is closed under arbitrary coproducts.
  \end{enumerate}

  Property (a) follows from the fact that $k[-]$ is a symmetric monoidal functor
  and both the source and target categories have the cartesian symmetric monoidal structure.
  Property (b) follows from the fact that $U$ is conservative and  
  Corollary \ref{corr:proconst-limits}.
  Property (c) holds because the composite $\Map_{\cCAlg_k}(k, k[\ph])$ preserves coproducts by Proposition \ref{prop:coprod-coalg}.

  \textbf{Step (1):} Eilenberg--MacLane spaces.\\  
  Let $X \coloneqq \Omega^\infty \Sigma^n \F_p\{S\}$ with $S \in \Set_*^\ainj$ and $n \ge 1$.
  Applying Proposition \ref{prop:key} we obtain a pullback square
  \begin{equation*}
    \begin{CD}
      k[\Omega^\infty \Sigma^n \F_p\{S\}] @>>> C(\Sigma^n k\{S\}) \\
      @VVV         @V{1-F}VV \\
      * @>>> C(\Sigma^n k\{S\}).
    \end{CD}
  \end{equation*}
  Applying $\Map_{\cCAlg_k}(k,-)$ to this pullback square and
  using the description of the Frobenius co-operation from Remark \ref{rmk:Frob-Q0} we obtain
  the pullback square
  \[ \begin{tikzcd}
    \Map_{\cCAlg_k}(k, k[\Omega^\infty \Sigma^n \F_p\{S\}]) \ar[r] \ar[d] &
    \Omega^\infty \Sigma^n k\{S\} \ar[d, "{( (-) - (-)^p ) \otimes \F_p\{S\}}"] \\
    * \ar[r] &
    \Omega^\infty \Sigma^n k\{S\}.
  \end{tikzcd} \]
  Using the fact that $k$ is separably closed,
  this pullback can be expressed as a sum over $S$ of copies of $\Omega^\infty$ of
  the Artin--Schreier fiber sequence
  \[ \F_p \to k \xrightarrow{x \mapsto x - x^p} k. \]
  We may now read off that $X \in \DD$.
  
  \textbf{Step (2):} Generalized Eilenberg--MacLane spaces.\\
  Let $X \coloneqq \prod K(\pi_n, n)$ with each $\pi_n \in \Mod_{\F_p}^\heartsuit$.
  If we let $X_j \coloneqq \prod_{n \leq j} K(\pi_n,n)$,
  then property (a) and step (1) imply that $X_j \in \DD$.
  We now apply property (b) to conclude that $X \in \DD$.
  For this it suffices to observe that
  limits of Postnikov towers commute with $k[-]$
  and that $\tau_{\leq j}k[X] \cong \tau_{\leq j}k[X_j]$ (this gives pro-constancy).
  
  
  \textbf{Step (3):} The connected case.\\
  Let $X \in \Spc_p^{\mathrm{nil}}$ be connected.
  Choose a base point in $X$, i.e., a lift of $X$ to $\Spc_*$.
  Write $X^\bullet$ for the cosimplicial object with
  $X^n = (\Omega^\infty \circ \F_p \otimes \Sigma^\infty)^{\circ (n+1)} X$.
  Since $X$ is nilpotent and $p$-complete,
  we know that $X \cong \lim X^\bullet$ \cite[Proposition VI.6.2]{bousfield1987homotopy}.
  Note that from step (2) we know that each $X^k \in \DD$.
  Using the base-points of the $X^k$ we obtain an equivalence of
  co-augmented cosimplicial objects
  \[ U(k[X^\bullet]) \to (k \{X^\bullet\}) \oplus k. \]
  The right-hand term is split, being obtained by scalar extension from $\F_p \{X^\bullet \}$ which is split \cite[\S4.7.3]{HA}, and hence has constant underlying pro-object (being a universal limit diagram \cite[Lemma 6.1.3.16]{HTT}).
  Applying property (b) we conclude.

  \textbf{Step (4):} The general case.
  Let $X \in \Spc_p^{\mathrm{nil}}$.
  We have \[ X = \coprod_{i \in \pi_0 X} X_i, \] with $X_i$ connected and nilpotent.
  The result is thus immediate from step (3) and property (c) of $\DD$.
\end{proof}



We now proceed to the task of determining the essential image of simply connected $p$-adic spaces in $\cCAlg_k$.

\begin{theorem} \label{thm:charact}
  Let $k$ be a separably closed field of characteristic $p$.
  The essential image of the functor
  \[ k[-] \colon (\Spc_p)_{\geq 2} \to \cCAlg_k \]
  is given by those $E \in \cCAlg_k$ such that
  $\pi_i(E) = 0$ for $i<0$, $\pi_0(E) = k$, $\pi_1(E) = 0$, and $\pi_i(E)$ is \solvable{} for every $i$.
\end{theorem}

The strategy for proving this theorem will be to inductively construct a map from $R$ to chains on a space which acts as the $n$'th term in an Adams resolution attached to $R$. Before the proof we need some preliminaries which ensure we have enough convergence to implement this strategy.

\begin{theorem}[Bousfield] \label{thm:bousfield-convergence}
\begin{enumerate}
\item Let $X_1 \leftarrow X_2 \dots \in \Spc_*$ be an inverse system of connected, pointed spaces.
  Assume that the homotopy pro-groups $\pi_i(X_\bullet)$ are pro-constant for every $i$.
  Then $k[\lim_i X_i] \wequi \lim_i k[X_i] \in \Mod_k$.
\item Let $X^\bullet \in \Fun(\Delta, \Spc_*)$ be such that each $X^n$ is connected and $R$-nilpotent, $R=\Z$ or $R=\F_p$.
  Suppose that the homotopy pro-groups of $\F_p[X^\bullet]$ are pro-constant, and trivial in degrees $\le 1$.
  Then for any $R$-algebra $S$, $S[\lim_\Delta X^\bullet] \wequi \lim_\Delta S[X^\bullet] \in \Mod_S$.
\item Let $F \to X \to B$ be a fiber sequence of simply connected, pointed spaces.
  Then the fiber sequence is preserved by the functor $k[\ph]: \Spc \to \cCAlg_k$.
\item Let $X_1 \leftarrow X_2 \dots$ be an inverse system of simply-connected, pointed, $p$-complete spaces.
  Assume that the homology pro-groups $\pi_i(\F_p[X_\bullet])$ are pro-constant.
  Then the limit is preserved by the functors $\Spc \to \cCAlg_k \to \Mod_k$.
\end{enumerate}
\end{theorem}
\begin{proof}
(1) Let $X = \lim X_\bullet$.
Then the map $X \to X_\bullet$ is an equivalence in the category of pro-truncated spaces, and this property is preserved by the functor $k[-]$.
Now materialize into $\Mod_k$.

(2) This is \cite[Theorem 3.6]{bousfield-homology-spectral-sequence}.

(3)
Rewrite the fiber sequence as a totalization and apply (2) with $R=\Z$, to find that the totalization is preserved by $k[-]: \Spc \to \Mod_k$.
Similarly rewrite the fiber sequence as a totalization in $\cCAlg_k$ and use Corollary \ref{corr:proconst-limits} to find that the limit is preserved by $\cCAlg_k \to \Mod_k$.

(4) Let $X_i^\bullet \in \Fun(\Delta, \Spc_*)$ denote the cotriple resolution of $X_i$ via reduced $\F_p$-homology, so that $\lim_\Delta X_i^\bullet \wequi X_i$ by \cite[Proposition VI.6.2]{bousfield1987homotopy}.
Let $X_\infty^\bullet = \lim_i X_i^\bullet$; note that this consists of $\F_p$-nilpotent spaces (in fact $\F_p$-modules).
Set $X = \lim_i X_i = \lim_\Delta X_\infty^\bullet$.
We must show that $k[X] \wequi \lim_i k[X_i]$.
The assumptions and Corollary \ref{corr:proconst-limits} imply that the limit on the right hand side can be computed in either $\cCAlg_k$ or $\Mod_k$.
Note that $X_i^0$ is pro-constant up to pro-truncation by assumption, and hence the same is true for $X_i^n$.
It follows (by (1)) that $k[X_\infty^n] \wequi \lim_i k[X_i^n]$.
Consequently, writing $T(\ph)$ for the totalization in pro-objects, we find that \[ T(k[X_\infty^\bullet]) \wequi \lim_i T(k[X_i^\bullet]) \wequi \lim_i k[X_i]. \]
Here the last $\lim_i$ is computed in pro-objects, and the last equivalence holds because each $k[X_i^\bullet]$ is pro-constant.
The final pro-object is pro-constant up to pro-truncation by assumption, whence (2) applies to $X_\infty^\bullet$.
This proves the result.
\end{proof}

\begin{proof}[Proof of Theorem \ref{thm:charact}.]
First note that the condition is necessary.
The non-obvious point is that $\pi_i(k[X])$ is \solvable; this was proved in \Cref{cor:pi-star-solvable}.

We now establish sufficiency.
Thus let $E \in \cCAlg_k$ which is connected and such that all it homotopy groups are \solvable.
We shall now inductively construct a sequence of simply-connected $p$-complete spaces $X_0 \leftarrow X_1 \leftarrow \dots$ together with compatible maps $E \to k[X_i]$.
It will be the case that $\pi_i(E) \to H_i(X_n,k)$ is injective for every $i$ and $n$.
We set $X_0 = \prod_{i \ge 2} K(\pi_i(E)^0, i)$.
This product is preserved by $\Spc \to \cCAlg_k$ and, using Lemma \ref{lemm:spacelike-homology-maps}, we obtain a canonical map $E \to k[X_0]$.
Since
\[ \pi_i(E) \cong k \otimes \pi_i(E)^0 \subseteq \pi_i(k[K(\pi_i(E)^0, i)]) \]
(the first isomorphism being true because $\pi_i(E)$ is \solvable)
we see that $E \to X_0$ indeed induces an injection on homotopy.
Now suppose $X_n$ has been constructed.
Let $C$ be the cofiber of $E \to k[X_n]$.
It follows from Lemma \ref{cor:pi-star-solvable} that the homotopy groups of $C$ are \solvable, and hence there is a simply connected space $C_0$ together with a map $C \to k[C_0]$ which is injective on homotopy groups.
Let $X_{n+1}$ be the fiber of $X_n \to C_0$.

We see inductively that $\pi_2(E) \wequi H_2(X_n, k)$ and hence $\pi_2(C) = 0$.
Thus by construction $\pi_2(C_0) = 0$ and $X_{n+1}$ is simply connected.
(Also $H_3(C_0, k) \iso \pi_3(C)$ receives a surjective map from $H_3(X_n, k)$, from which one deduces via e.g. a Serre spectral sequence that $H_2(X_{n+1}, k) \iso H_2(X_n, k)$.)
Theorem \ref{thm:bousfield-convergence}(3) implies that the fiber sequence $X_{n+1} \to X_n \to C_0$ is preserved when passing to coalgebras.
By construction, the composite $E \to k[X_n] \to C \to C_0$ is null and hence $E \to k[X_n]$ lifts along $X_{n+1} \to X_n$.
In particular $E \to k[X_{n+1}]$ induces an injection on homotopy groups and the construction can continue.

We now obtain exact sequences
\[ 0 \to \pi_*E \to H_*(X_n, k) \to H_*(C_0,k). \]
Since the composite $X_{n+1} \to X_n \to C_0$ is null, we see that the image of $H_*(X_{n+1}, k) \to H_*(X_n, k)$ lies inside the image of $\pi_*E$.
On the other hand (by the existence of $E \to k[X_{n+1}]$) the map also surjects onto that image.
It follows that $\pi_iE \to H_i(X_\bullet, k)$ is a pro-isomorphism for every $i$.
We may hence apply Theorem \ref{thm:bousfield-convergence}(4) to conclude that $k[\lim_i X_i] \wequi \lim_i k[X_i]$, and this limit is computed on underlying $k$-modules.
We thus obtain a map \[ E \to \lim_i k[X_i] \wequi k[\lim_i X_i], \] which is an isomorphism.
This concludes the proof.
\end{proof}

\bibliographystyle{alpha}
\bibliography{bibliography}

\end{document}

%% file: preamble2.tex
\usepackage{amsmath}
\usepackage{amssymb}
\usepackage{amsthm}
\usepackage{amscd}
\usepackage{enumerate}
\usepackage{scalerel}
\usepackage[pdfusetitle,unicode,hidelinks]{hyperref}
\usepackage{cleveref}
\usepackage{bbm}
\usepackage{etoolbox}

\usepackage[utf8]{inputenc}
\usepackage[T1]{fontenc}

\newtheorem{proposition}{Proposition}
\newtheorem{corollary}[proposition]{Corollary}
\newtheorem{lemma}[proposition]{Lemma}
\newtheorem{theorem}[proposition]{Theorem}

\newtheorem*{conjecture*}{Conjecture}
\newtheorem*{theorem*}{Theorem}
\newtheorem*{corollary*}{Corollary}
\newtheorem*{proposition*}{Proposition}
\newtheorem*{lemma*}{Lemma}

\theoremstyle{definition}
\newtheorem{definition}[proposition]{Definition}
\newtheorem{construction}[proposition]{Construction}
\newtheorem{assumption}[proposition]{Assumption}
\newtheorem*{definition*}{Definition}
\newtheorem*{construction*}{Construction}
\newtheorem{remark}[proposition]{Remark}
\newtheorem*{remark*}{Remark}
\newtheorem{question}[proposition]{Question}
\newtheorem{example}[proposition]{Example}
\newtheorem*{example*}{Example}
\newtheorem{rec}[proposition]{Recollection}
\newtheorem{ntn}[proposition]{Notation}

\AtBeginEnvironment{definition}{\pushQED{\qed}}
\AtEndEnvironment{definition}{\popQED\enddefinition}
\AtBeginEnvironment{construction}{\pushQED{\qed}}
\AtEndEnvironment{construction}{\popQED\endconstruction}
\AtBeginEnvironment{assumption}{\pushQED{\qed}}
\AtEndEnvironment{assumption}{\popQED\endassumption}
\AtBeginEnvironment{definition*}{\pushQED{\qed}}
\AtEndEnvironment{definition*}{\popQED\enddefinition*}
\AtBeginEnvironment{construction*}{\pushQED{\qed}}
\AtEndEnvironment{construction*}{\popQED\endconstruction*}
\AtBeginEnvironment{remark}{\pushQED{\qed}}
\AtEndEnvironment{remark}{\popQED\endremark}
\AtBeginEnvironment{remark*}{\pushQED{\qed}}
\AtEndEnvironment{remark*}{\popQED\endremark*}
\AtBeginEnvironment{question}{\pushQED{\qed}}
\AtEndEnvironment{question}{\popQED\endquestion}
\AtBeginEnvironment{example}{\pushQED{\qed}}
\AtEndEnvironment{example}{\popQED\endexample}
\AtBeginEnvironment{example*}{\pushQED{\qed}}
\AtEndEnvironment{example*}{\popQED\endexample*}
\AtBeginEnvironment{rec}{\pushQED{\qed}}
\AtEndEnvironment{rec}{\popQED\endrec}
\AtBeginEnvironment{ntn}{\pushQED{\qed}}
\AtEndEnvironment{ntn}{\popQED\endntn}

\newcommand{\Z}{\mathbb{Z}}

\newcommand{\N}{\mathbb{N}}

\newcommand{\F}{\mathbb{F}}
\def\E{\mathbb{E}}

\let\scr=\mathcal

\newcommand{\1}{\mathbbm{1}}

\DeclareMathOperator*{\colim}{colim}

\let\lim=\relax
\DeclareMathOperator*{\lim}{lim}
\def\Map{\mathrm{Map}}

\def\map{\mathrm{map}}
\def\CAlg{\mathrm{CAlg}}

\def\Ind{\mathrm{Ind}}

\def\Spc{\mathcal{S}\mathrm{pc}{}}
\def\Fin{\cat F\mathrm{in}}
\def\Fun{\mathrm{Fun}}

\newcommand{\Spec}{\mathrm{Spec}}

\newcommand{\iso}{\cong}
\newcommand{\wequi}{\iso}
\newcommand{\Mod}{\text{-}\mathcal{M}\mathrm{od}}

\newcommand{\heart}{\heartsuit}

\newcommand{\Hom}{\operatorname{Hom}}

\def\op{\mathrm{op}}

\let\cat=\mathrm

\def\ph{\mathord-}

\def\poly{\mathrm{poly}}